\documentclass[letterpaper,12pt]{article}

\usepackage{amsmath}
\allowdisplaybreaks

\usepackage{enumitem}
\usepackage{microtype,booktabs}
\usepackage{amsmath,amssymb,amsfonts,amsthm,bm}
\usepackage{authblk}
\usepackage{caption}
\usepackage{subcaption}
\usepackage{mathrsfs,dsfont}
\usepackage{indentfirst}
\usepackage[pdftex]{color,graphicx}
\usepackage[pdftex,bookmarks,unicode,colorlinks,linkcolor=blue]{hyperref}
\usepackage{tikz}
\usepackage{makecell}
\usepackage{todonotes}
\usepackage{comment}
\usepackage{xcolor}
\definecolor{grey}{rgb}{0.5,0.5,0.5}
\usepackage{amsmath}

\newtheorem{Lemma}{Lemma}[section]

\title{\Large Metastable Transitions in  Dynamical Systems with both Time-varying Perturbations and Degenerate Noise}

\date{}
\author[$\dag$]{\small Hanru Zou}
\author[$\dag$]{\small Hongjun Gao}
\author[$\dag$]{{Pingyuan Wei}\thanks{Corresponding authors\par \textit{\;\;Email:} pwei@seu.edu.cn (Pingyuan Wei), yingchao1993@xjtu.edu.cn (Ying Chao)}}
\author[$\ddag$]{\small Ying Chao$^*$}
 \affil[$\dag$]{School of Mathematics, Southeast University, Nanjing 211189, China.}
 \affil[$\ddag$]{\footnotesize School of Mathematics and Statistics, Xi'an Jiaotong University, Xi'an 710049, China.}

\newtheorem{thm}{Theorem}[section]
\newtheorem{theorem}{Theorem}[section]

\theoremstyle{remark}

\newtheorem{remark}[thm]{Remark}
\theoremstyle{definition}

\begin{document}

\maketitle

\begin{abstract}
This paper investigates the persistence of maximum likelihood paths in degenerate stochastic differential systems and quantifies how small periodic perturbations modulate the metastable transition rate. Within the Freidlin--Wentzell large deviation framework, we reformulate the variational problem for MLPs as a Hamiltonian system via a partial Legendre transform. Under hyperbolicity and transversality conditions, we prove, using a geometric Melnikov method adapted to general time-dependent perturbations, that the corresponding heteroclinic connections persist for sufficiently small perturbations. For the periodic case, we derive a closed-form explicit expression for the rate change to first order in the forcing amplitude. Two illustrative examples are presented.
\end{abstract}


Maximum likelihood paths, Freidlin--Wentzell large deviations, degenerate stochastic systems, periodic perturbations, geometric Melnikov method.

\renewcommand{\theequation}{\thesection.\arabic{equation}}
\setcounter{equation}{0}

\section{Introduction}\label{sec:intro}
Metastable transitions are a fundamental mechanism in stochastic dynamical systems, describing rare transitions between long-lived stable states induced by small random fluctuations. In the deterministic setting, a trajectory typically remains trapped in the basin of attraction of a stable equilibrium or attractor. In the presence of noise, however, the system may escape from one basin and reach another metastable state over a sufficiently long time scale. Such transitions arise in a wide range of physical, geophysical, and biological models
\cite{Ragone2018,Dematteis2018,Chaudhury2012}. In this paper, we consider a general non-autonomously forced degenerate stochastic mechanical system, aiming to study the persistence of maximum likelihood paths (MLPs) and to quantify how periodic perturbations change the rate of metastable transitions. This idea is inspired, for example, by the investigation of periodically perturbed Markov jump processes in chemical and
epidemiological applications~\cite{Billings2018}, by stochastic resonance~\cite{Lucarini2019}, by the use of nongradient forcing, which can be interpreted as an irreversible component, just as time-dependent perturbations can be interpreted similarly, to change transition rates~\cite{Heymann2008}, and by many successes in controlling deterministic systems using
periodic perturbations~\cite{Surappa2018,Xie2019,Chen2019}.

More specifically, consider the following \textit{degenerate SDE}:
\begin{equation}\label{dSDE}
\begin{cases}
\displaystyle
dX_t = F_1(X_t, Y_t) dt, \\
\displaystyle
dY_t = \big[ F_2(X_t, Y_t) + \varepsilon g(t,X_t, Y_t) \big] dt + \sqrt{\mu} \, \sigma (X_t,Y_t) \, dB_t,
\end{cases}
\end{equation}
where $(X_t,Y_t)^{\top}\in\mathbb{R}^n\times\mathbb{R}^m$, $t\in[a,b]$, $\varepsilon, \mu > 0$ are parameters, $F_1: \mathbb{R}^{n}\times \mathbb{R}^{m} \to \mathbb{R}^n$ and $F_2: \mathbb{R}^{n}\times \mathbb{R}^{m} \to \mathbb{R}^m$ are drift functions, $g:[a,b]\times \mathbb{R}^{n}\times \mathbb{R}^{m}\to \mathbb{R}^{m}$ is a time-varying  perturbation, $\sigma:\mathbb{R}^{n}\times \mathbb{R}^{m}\to\mathbb{R}^{m\times m}$ is the noise coefficient matrix,  and $B_t$ is an $m$-dimensional standard Brownian motion.

In \eqref{dSDE}, we assume $0<\mu\ll\varepsilon$, where the two small parameters $\varepsilon$ and $\mu$ represent the strengths of the periodic forcing and the noise perturbation, respectively.  When $\varepsilon\ne0,\;\mu=0$ (in the deterministic version), there exist two stable orbits $z_a^\varepsilon(t)$ and $z_b^\varepsilon(t)$ separated by an unstable orbit $z_u^\varepsilon(t)$, bifurcating from $z_a,\;z_b$ and $z_u$, respectively, where $z_u$ lies on the separatrix between the basins of attraction of the two stable states~\cite{Tao2018}.  Here $z_i, i=a,b,u$, denotes an equilibrium of the unperturbed system; no additional dynamical classification is required. When $\mu\ne0$, metastable transitions from the deterministic stable periodic orbits toward the unstable orbit, which are impossible without noise, occur due to small noise perturbations. Furthermore, the time-dependent forcing $g(t,X,Y)$ is added to quantify such transitions in the regime of small noise.

To quantify the metastable transition under small noise, we rely on the Freidlin–Wentzell large deviation theory, which provides a principled asymptotic characterization of rare-event probabilities. The leading-order transition probability from one stable orbit to another is determined by the minimum action, and the corresponding minimizers are referred to as the maximum likelihood paths. Since the transition time is typically not prescribed in metastable regimes, it is natural to consider an additional optimization over the time horizon; in many settings the relevant minimizers emerge in the infinite-time limit, which motivates formulating the problem for MLPs on an unbounded time interval. This infinite-horizon viewpoint naturally links the MLPs to heteroclinic type solutions of the associated Euler–Lagrange (equivalently, Hamiltonian) system, providing a practical and principled route to compute transition rates and most probable transition mechanisms (as clarified in the next section).

Over the past two decades, continuous and significant mathematical and experimental efforts have been devoted to understanding transitions and escapes in periodically forced systems with small noise. Smelyanskiy et al.~\cite{Smelyanskiy1997} and Dykman et al.~\cite{Dykman1997} proved that escape probabilities can be significantly changed even by a comparatively weak force. Agudov and Spagnolo ~\cite{Agudov2001}, and Dubkov et al.~\cite{Dubkov2004} studied the effect of noise-enhanced stability of periodically driven metastable states. Chen et al.~\cite{Chen2019} identified an MLP
under periodic forcing for the case of finite noise. Jiang et al.~\cite{Jiang2022} computed numerically the MLP in the eutrophicated lake ecosystem under Gaussian white noise and periodic forcing.   Seminal results on the noise-activated escape rate of a second-order and under-damped dynamical system can be found in~\cite{Dykman2001,Smelyanskiy1997} for the case of a single particle under linear additive driving, i.e., $g(t,x,y)=g(t)$.  Chao and Tao~\cite{Chao2022} studied how periodic perturbations affect the rate of metastable
transitions in weakly noisy underdamped kinetic Langevin systems. Their framework
applies to multiparticle/higher-dimensional mechanical systems and allows nonlinear
periodic forcings. Recently,  Chao, Duan and Wei~\cite{wei2024} considered noise-induced metastable transitions in a
one-dimensional overdamped Langevin system under small periodic forcing, presented conditions under which the MLPs connecting periodic orbits survive, and obtained an $\omega$ -dependent rate change function. For non-gradient systems, Tao~\cite{Tao2018} studied the hyperbolic periodic orbits and small-noise-induced metastable transitions and Fleurantin et al.~\cite{Emmanuel2023} focused on the distinguished most probable escape paths crossing the periodic orbits without gradient structure.

To extend the above studies, our work is devoted to identifying the intrinsic structural features of the underlying deterministic system, present concrete conditions under which the MLPs connecting time-dependent orbits survive, and then obtain an $\omega$-dependent rate change function.

Unlike previous work, our setting is no longer restricted to specific structural assumptions and includes a broader range of dynamical structures, such as orthogonal systems and certain nonlinear systems.
Moreover, we consider degenerate SDEs  under locally weak
monotonicity conditions and Lyapunov conditions. In addition, in this paper, the persistence of MLPs for general higher-dimensional degenerate systems under small time-dependent forcing is treated explicitly using dynamical-systems techniques, such as geometric Melnikov methods. In the absence of time periodicity, the usual Poincaré-map formulation and the associated lobe-dynamics description are no longer available. Therefore, we further develop a geometric Melnikov method under general time-dependent perturbations. By fixing a time slice and parametrizing the unperturbed heteroclinic orbit, we measure the splitting the perturbed stable and unstable manifolds by projecting their displacement onto the conormal covector along the unperturbed orbit, leading to a Melnikov integral criterion for the persistence of heteroclinic orbits.

The main idea of this paper is the following: 

$\bullet$ First, we formulate the Freidlin--Wentzell framework for the
non-autonomous degenerate stochastic system, which provides the action
functional and the exponential characterization of metastable transition rates.

$\bullet$ Then we transform the Euler-Lagrange problem into a higher-dimensional perturbation Hamiltonian system, which allows us to characterize the perturbed MLPs using an extended geometric Melnikov method~\cite{Melnikov1963,Priyankara2022}. 

$\bullet$ Finally, we use a linear-theory calculation inspired by~\cite{Assaf2008,Dykman2001,Smelyanskiy1997} to approximate the minimum of the Freidlin–Wentzell action to first order in $\varepsilon$. Thus, we derive a closed-form explicit expression that characterizes how a small, generic, nonlinear periodic forcing affects the metastable transition rate in the general SDE \eqref{dSDE}.

The rest of the paper is arranged as follows. In section \ref{sec:main1}, we define the Freidlin–Wentzell large deviation principle for degenerate SDEs and consider the metastable transition based on large deviation principle. In section \ref{sec3}, we reformulate the variational problem as a Hamiltonian formalism.  And later we show the persistence of heteroclinic connections for the higher-dimensional Hamiltonian system under small perturbations in  section \ref{sec4}. Then we  utilize an equivalent description of Freidlin--Wentzell action and further characterize the transition rate variations under periodic perturbations in section \ref{sec5}. Finally, numerical experiments of degenerate SDEs are exhibited in section \ref{sec6} and conclusions follow in section \ref{sec7}.

\renewcommand{\theequation}{\thesection.\arabic{equation}}
\setcounter{equation}{0}
 
\section{Freidlin--Wentzell Framework for Degenerate Stochastic Systems}
\label{sec:main1}
In this section, we first state the basic assumptions of degenerate stochastic differential equations and then introduce the Freidlin--Wentzell large deviation theory. Finally, we use this framework to study metastable transitions between stable and unstable states.
\subsection{Degenerate Stochastic Differential Equations and Freidlin--Wentzell Large Deviation}
\label{sec:degenerate}

Let $z_t = (x_t, y_t)^\top \in \mathbb R^n\times\mathbb R^m$. The system \eqref{dSDE} can be written as:
\begin{equation}\label{dSDE2}
    dz_t = A^\varepsilon(t, z_t) dt + \sqrt{\mu} \, \tilde{\sigma}(z_t) \, dB_t,
\end{equation}
with drift and diffusion:
$$
A^\varepsilon(t,z) :=
\begin{pmatrix}
F_1(x,y) \\
F_2(x,y) + \varepsilon g(t,x,y)
\end{pmatrix}, \quad
\tilde{\sigma}(z) :=
\begin{pmatrix}
0 \\
\sigma(x,y)
\end{pmatrix}.
$$
Denote the covariance matrix by $\Sigma(z):=\sigma(z)\sigma(z)^\top\in\mathbb{R}^{m\times m}$. Hence, $\tilde\sigma\tilde\sigma^\top=\mathrm{diag}(0,\Sigma)$.

\noindent\textbf{Assumption \textbf{(A1)}}
\begin{enumerate}
    \item[($i$)] \textit{(Integrability condition on $g$)} For arbitrary $R>0$, 
    \begin{align}
        \int_a^b\sup_{|z|\leq R}\left|g(s,z)\right|ds<\infty.\notag
    \end{align}
    \item[($ii$)] \textit{(Local monotonicity conditions)} Let $\delta_0\in(0,1)$. For arbitrary $R>0$, if $|z_1|\vee|z_2|\leq R$, $|z_1-z_2|\leq \delta_0$, there exists $L_R>0$ and a nonnegative integrable function $\eta_ R(t)$, $t\in [a,b]$, such that 
    \begin{equation}
        2\left\langle z_1-z_2,
        \begin{pmatrix}
        F_1(z_1)-F_1(z_2) 
        \\F_2(z_1)-F_2(z_2)
        \end{pmatrix}
        \right\rangle+ \big\|{\sigma}(z_1)-{\sigma}(z_2)\big\|^2 \leq L_R |z_1-z_2|^2\notag
    \end{equation}
    and
    \begin{equation}
        2\left\langle z_1-z_2,
        g(s,z_1)-g(s,z_2)
        \right\rangle \leq \eta_R(s) |z_1-z_2|^2,\;\; s\in [a,b].\notag
    \end{equation}
    Here $\langle \cdot, \cdot \rangle$ denotes the Euclidean inner product. 
    \item[($iii$)] \textit{(Lyapunov conditions)} There exists a Lyapunov function $\mathcal U\in C^2(\mathbb{R}^{n}\times\mathbb R^m;\mathbb{R}_+)$ and $\theta>0$, $\vartheta>0$ such that 
    \begin{equation}
        \lim_{|z|\to +\infty} \mathcal U(z)=+\infty,\notag
    \end{equation}
    \begin{equation*}
        \left\langle 
        F(z), \nabla \mathcal U(z) \right\rangle+ \frac{\theta}{2} \text{Trace} \big({\sigma}(z)^\top \nabla_{yy}^2 \mathcal U(z){\sigma}(z) \big)+\frac{|{\sigma}(z)\cdot \nabla_y \mathcal U(z)|^2}{\vartheta \mathcal U(z)}\leq C(1+\mathcal U(z)),
     \end{equation*}
    \begin{equation}
        \left\langle g(s,z),\nabla \mathcal U(z)
         \right\rangle\leq \eta_\mathcal U(s)(1+\mathcal U(z)),\;\;s\in[a,b],\notag
    \end{equation}
    and
    \begin{equation}
        \text{Trace} \big({\sigma}(z)^\top \nabla_{yy}^2 \mathcal U(z){\sigma}(z) \big)\geq 0,\notag
    \end{equation} where $F(z)=\begin{pmatrix}
        F_1(z)
        \\F_2(z)
        \end{pmatrix}$.
    Here $C>0$ is a constant and $\eta_\mathcal U(t)$, $t\in [a,b]$, is a nonnegative integrable function.
    \item[($iv$)] \textit{(Uniform non-degeneracy)} The diffusion matrix $\Sigma=\sigma\sigma^\top$ is uniformly positive definite: there exists $\lambda_0>0$ such that
    $$
    v^\top\Sigma(z)v \ge \lambda_0 |v|^2,\qquad \forall z\in\mathbb R^n\times\mathbb R^m,\ \forall v\in\mathbb R^m.
    $$

\end{enumerate}

By~\cite[Proposition 2.1]{Wang2024}, under Assumption \textbf{(A1)}, there exists a unique strong solution of equation \eqref{dSDE}. Moreover, $\Sigma(z)$ is invertible for all $z$ due to its uniform positive definiteness. For fixed $\varepsilon>0$ and  any $\phi=(\phi_x,\phi_y)\in C([a,b],\mathbb{R}^{m+n})$, we define the \textit{(Freidlin--Wentzell) action functional} as
\begin{align}\label{FW-deg}
S^\varepsilon_{\text{FW}}[\phi] 
=& \frac{1}{2} \int_{a}^{b} \big\| \dot{\phi}_y(t) - F_2(\phi_x,\phi_y) - \varepsilon g(t,\phi_x,\phi_y) \big\|^2_{\Sigma(\phi)^{-1}} dt,
\end{align}
 where the weighted norm is defined as $\|v\|^2_M = v^\top M v$ for any vector $v$ and positive definite matrix $M$. Then we further have the following large deviation principle based on~\cite[Theorem 2.1]{Wang2024}: 
\begin{Lemma}[Freidlin–Wentzell large deviation principle]
    For $\mu>0$, let $Z^\mu$ be the solution to equation \eqref{dSDE2}. If Assumption \textbf{(A1)} holds, then the family $\{Z^\mu\}_{\mu>0}$ satisfies a large deviation principle on the space $C([a,b],\mathbb{R}^{m+n})$ with the rate function $\mathcal{I}^\varepsilon_{\text{FW}}: C([a,b],\mathbb{R}^{m+n})\to [0,\infty]$, where  
    \begin{equation}
        \mathcal{I}^\varepsilon_{\text{FW}}=\inf_{\phi\in C([a,b],\mathbb{R}^{m+n});\; {\dot{\phi}_x(t) = F_1(\phi_x(t),\phi_y(t))}}S^\varepsilon_{\text{FW}}[\phi],
    \end{equation}
    with the convention $\inf\emptyset=\infty$ and the action functional $S^\varepsilon_{\text{FW}}$ given in \eqref{FW-deg}.
\end{Lemma}

\subsection{Metastable Transitions Based on Freidlin–Wentzell Large Deviation Principle}

In this subsection, we aim to investigate metastable transitions for the degenerate system \eqref{dSDE}, with particular emphasis on how time-varying perturbations $g(t,x,y)$ and the degenerate noise structure influence both the transition pathways and rates between metastable states. Our analysis leverages the geometric reformulation of the Freidlin–Wentzell action principle to characterize the maximum likelihood transition mechanisms in this degenerate setting.

\noindent\textbf{Assumption \textbf{(A2)}}
\begin{itemize}
\item[($i$)] \textit{(Parameter hierarchy)} The parameters satisfy the asymptotic regime: $0 < \mu < \varepsilon \ll 1$. 

\item[($ii$)] \textit{(Metastable structure)} In the unperturbed deterministic case ($\mu = 0$, $\varepsilon = 0$), the system $\dot{z} = F(z)$ possesses three hyperbolic equilibrium points: two stable equilibria $z_a$ and $z_b$, and a saddle (unstable) equilibrium $z_u$ lying on the separatrix between their basins of attraction.
\end{itemize}

This parameter hierarchy ensures that we can first establish a large deviation principle (as $\mu \to 0$), then perform asymptotic analysis with respect to $\varepsilon$. In the presence of noise, the stable equilibria become \textit{metastable states}, enabling transitions between them. However, the time-dependent nature of the perturbation $g(t,x,y)$ introduces additional challenges compared to the autonomous case, as the metastable states and transition mechanisms become time-dependent.

\begin{Lemma}[Persistence under time-dependent forcing]
\label{Lemma:persistence}
Under Assumption \textbf{(A2)}, for sufficiently small $\varepsilon > 0$, the deterministic system ($\mu=0$)
\begin{equation*}\label{deterministic}
\dot{z} = A^\varepsilon(t,z) = F(z) + \varepsilon \begin{pmatrix} 0 \\ g(t,z) \end{pmatrix},
\end{equation*}
exhibits two stable orbits $z_a^\varepsilon(t)$ and $z_b^\varepsilon(t)$ and an unstable orbit $z_u^\varepsilon(t)$ on the time interval $[a,b]$, which are continuous in $\varepsilon$ and, as $\varepsilon \to 0$, converge uniformly to $z_a$, $z_b$, and $z_u$ respectively. Moreover, the stability of $z_a^\varepsilon(t)$ and $z_b^\varepsilon(t)$ and the instability of $z_u^\varepsilon(t)$ are preserved for small $\varepsilon$.

In the special case where $g(t,z)$ is $\tau$-periodic in $t$, these orbits become $\tau$-periodic.
\end{Lemma}

The proof of Lemma~\ref{Lemma:persistence} follows from standard arguments in dynamical systems theory. The existence of the perturbed orbits is guaranteed by the implicit function theorem. The stability properties can be characterized via Lyapunov exponents (in the general time-dependent case) or Floquet multipliers (in the periodic case)~\cite{Meiss2007}, which constitute continuous perturbations of the eigenvalues of $DF(z)$ at the original equilibria.

Note that in the small noise regime, metastable transitions typically occur over an exponentially long time scale, which effectively requires the time horizon $b - a \to \infty$. It is therefore natural to consider the limiting case where the initial time $a = -\infty$ and the terminal time $b = +\infty$. To study transitions on the whole time axis we pass to $t\in\mathbb R$ and extend the action by
\begin{equation*}
    S^\varepsilon_{\mathrm{FW}}[\phi]=\frac12\int_{-\infty}^{+\infty}
  \big\|\,\dot\phi_y(t)-F_2(\phi(t))-\varepsilon g(t,\phi(t))\,\big\|^2_{\Sigma(\phi(t))^{-1}}\,dt,
\end{equation*}
whenever the integral is finite, and set $S^\varepsilon_{\mathrm{FW}}[\phi]=+\infty$ otherwise. Let $z_a^\varepsilon(t)$ and $z_b^\varepsilon(t)$ denote the perturbed (meta)stable orbits as introduced in Lemma~\ref{Lemma:persistence}. We then define the admissible path space by
\begin{align}
  \mathscr A
  :=\Big\{\,\phi=(\phi_x,\phi_y)^{\top}\in AC_{\mathrm{loc}}(\mathbb R;\mathbb R^n\times\mathbb R^m):\ 
  \dot\phi_x(t)=F_1(\phi_x(t),\phi_y(t))\ \text{for a.e. }t,&\notag\\
  \lim_{t\to -\infty}(\phi(t)-z_a^\varepsilon(t))=0,\;
   \lim_{t\to +\infty}(\phi(t)-z_b^\varepsilon(t))=0,\;S^\varepsilon_{\mathrm{FW}}[\phi]<\infty&
  \Big\},\notag
\end{align}
where $AC_{\mathrm{loc}}(\mathbb R;\mathbb R^n\times\mathbb R^m)$ denotes the space of all locally absolutely continuous functions, which is equivalent to the local Sobolev space $W^{1,1}_{\mathrm{loc}}(\mathbb{R};\mathbb R^n\times\mathbb R^m)$.

Based on the Freidlin–Wentzell large deviation principle, to analyze transitions between the metastable orbits $z_a^\varepsilon(t)$ and $z_b^\varepsilon(t)$, the central idea is to identify a deterministic path: commonly referred to as the \textit{maximum likelihood path} or \textit{instanton}—around which the diffusion process is highly concentrated. More precisely, one aims to characterize the probability of such rare transitions through the asymptotic relation
\begin{equation}
    \mathbb{P}\left( \sup_{t \in (-\infty,+\infty)} \left| Z_t^{\mu} - \phi(t) \right| < \delta \right) \asymp \exp\left( -\frac{1}{\mu} S_{FW}^\varepsilon[\phi] \right), \quad\mu \to 0,
\end{equation}
for any $\phi \in \mathscr{A}$, where $\delta > 0$ denotes the radius of a small neighborhood around the deterministic reference path $\phi(t)$ within which the stochastic trajectory $Z_t^{\mu}$ remains localized. The minimum action path $\phi^\ast$ is then defined by
\begin{align}\label{MLP}
S^\varepsilon_{\text{FW}}[\phi^\ast] = \inf_{\phi \in \mathscr{A}} S^\varepsilon_{\text{FW}}[\phi],
\end{align} and represents the maximum likelihood path in the small-noise limit $\mu \to 0$.

For convenience, we denote by $\mathcal R_\varepsilon$ the transition rate from $z_a^\varepsilon(t)$ to $z_b^\varepsilon(t)$. The time-dependent Freidlin--Wentzell theory developed above, together with existing physical studies and numerical evidence~\cite{ Dykman1997,Dykman2001,Chen2016,Chao2022}, suggests that, as $\mu\to 0$,
\begin{equation*}
    \mathcal R_\varepsilon
    \asymp \exp\!\left(-\frac{1}{\mu}S_{FW}^\varepsilon[\phi^\ast]\right).
\end{equation*}


\section{Variational and Hamiltonian Formulations}\label{sec3}

In this section, we connect the concept of the MLPs with heteroclinic connections in phase space by employing the variational principle and Legendre transform. This then allows us to address the problem using techniques from Hamiltonian perturbation analysis.

\subsection{Constrained Variational Principle and Euler--Lagrange Equations}

To better understand the minimizer $\phi^\ast(t)$ (i.e., the MLP) of the Freidlin--Wentzell action functional $S^\varepsilon_{\text{FW}}$, we start with the corresponding variational formulation. Treating the $x$--dynamics as a constraint, we introduce a time-dependent Lagrange multiplier (adjoint) $\lambda(t)\in\mathbb R^n$ and define the augmented, time-dependent Lagrangian by
\begin{align}\label{L}
\mathcal L^\varepsilon(\phi, \dot{\phi}, \lambda, t)
= \frac12 \big(r^\varepsilon(t)\big)^{\top} \Sigma(\phi(t))^{-1} r^\varepsilon(t)
+ \lambda(t)^{\top} \big( \dot{\phi}_x(t) - F_1(\phi(t))\big),
\end{align}
where $r^\varepsilon(t):=\dot\phi_y(t)-F_2(\phi(t))-\varepsilon g(t,\phi(t)).$ Accordingly, the augmented action functional takes the form
$$
\mathcal J^\varepsilon[\phi,\lambda]
=\int_{-\infty}^{\infty}\mathcal L^\varepsilon(\phi,\dot\phi,\lambda,t)\,dt
= S^\varepsilon_{\text{FW}}[\phi]+\int_{-\infty}^{\infty}\lambda^{\top}(t)\big(\dot{\phi}_x(t)-F_1(\phi(t))\big)\,dt.
$$

The stationary conditions of $\mathcal J^\varepsilon$ are determined by the Euler--Lagrange equations with respect to $\phi=(\phi_x,\phi_y)$, where $\lambda$ is treated as an additional unknown satisfying its own variation equation. In the standard variational form,
\begin{equation}\label{EL}
    \frac{\delta\mathcal J^\varepsilon}{\delta\phi}
=\frac{\partial\mathcal L^\varepsilon}{\partial\phi}
-\frac{d}{dt}\frac{\partial\mathcal L^\varepsilon}{\partial\dot\phi}
=0,
\end{equation}
we obtain the componentwise system:
\begin{equation*}\label{ELL}
\begin{cases}
\displaystyle
\dot\phi_x = F_1(\phi),\\[0.4em]
\dot\phi_y = F_2(\phi) + \varepsilon g(t,\phi) + r^\varepsilon,\\[0.4em]
\displaystyle
\dot\lambda
= -\big(\partial_{\phi_x}F_2+\varepsilon\,\partial_{\phi_x}g\big)^{\top}\Sigma^{-1}r^\varepsilon
+\tfrac12\,(r^\varepsilon)^{\top}\big(D_{\phi_x}\Sigma^{-1}\big)r^\varepsilon
-(\partial_{\phi_x}F_1)^{\top}\lambda,\\[0.6em]
\displaystyle
\frac{d}{dt}\!\big(\Sigma^{-1}(\phi)\,r^\varepsilon\big)
= -\big(\partial_{\phi_y}F_2+\varepsilon\,\partial_{\phi_y}g\big)^{\top}\Sigma^{-1}r^\varepsilon
+\tfrac12\,(r^\varepsilon)^{\top}\big(D_{\phi_y}\Sigma^{-1}\big)r^\varepsilon
-(\partial_{\phi_y}F_1)^{\top}\lambda.
\end{cases}
\end{equation*}
Here $D_{\phi_x}\Sigma^{-1}$ and $D_{\phi_y}\Sigma^{-1}$ denote the derivatives of $\Sigma^{-1}(\phi)$ with respect to $\phi_x$ and $\phi_y$, respectively, and contractions such as $r^{\top} (D_{\phi_x}\Sigma^{-1}(\phi)) r$ are understood in the natural index sense. The above system represents the Euler--Lagrange equations for the constrained variational problem associated with $S_{\text{FW}}^{\varepsilon}$.

The minimizer $\phi^\ast(t)$ satisfies the (heteroclinic-type) boundary conditions
$$
\lim_{t\to -\infty}\big(\phi^*(t)-z_a^\varepsilon(t)\big)=0,\qquad
\lim_{t\to +\infty}\big(\phi^*(t)-z_b^\varepsilon(t)\big)=0,
$$
together with the transversality condition
$$
\lim_{t\to\pm\infty}\lambda(t)=0,
$$
which guarantees the vanishing of the boundary terms arising from the integration by parts in the variational derivation. We remark that, although each MLP satisfies the Euler–Lagrange equations, the solution is generally not unique.
Such nonuniqueness arises from time translation invariance and the possible coexistence of multiple transition channels between metastable states.

\subsection{Hamiltonian Structure via Partial Legendre Transform}

We now convert the Lagrangian problem to the Hamiltonian framework by introducing the generalized momenta
\begin{align}\label{L_psi}
    \psi_x := \frac{\partial \mathcal L^\varepsilon}{\partial \dot{\phi}_x}, \qquad
    \psi_y := \frac{\partial \mathcal L^\varepsilon}{\partial \dot{\phi}_y}.
\end{align}
Since $\mathcal L$ is linear in $\dot{\phi}_x$, we have $\psi_x = \lambda(t)$ and $\dot{\phi}_x$ cannot be expressed as a function of $\psi_x$. In contrast, the $y$--block satisfies the Legendre nondegeneracy condition, $\partial^2 \mathcal L^\varepsilon / \partial \dot{\phi}_y^2 = \Sigma^{-1}(\phi)$, which is invertible by Assumption \textbf{(A1)}(iv). Therefore, $\dot{\phi}_y$ can be locally expressed in terms of $(t, \phi_x, \phi_y, \psi_y)$ as
\begin{align*}
    \dot{\phi}_y = \Sigma \psi_y + F_2(\phi) + \varepsilon g(t,\phi).
\end{align*}  
Following~\cite{Arnold1989}, the Hamiltonian is then defined via the (partial) Legendre transform in the $y$--variables:
\begin{align}\label{Hf}
\begin{split}
   H_\varepsilon(t,\phi_x,\phi_y,\psi_x,\psi_y)
       &= \psi_x^{\top} \dot{\phi}_x + \psi_y^{\top} \dot{\phi}_y - \mathcal L^\varepsilon(\phi, \dot{\phi}, \lambda, t) \\
       &= \frac{1}{2} \psi_y^\top \Sigma \psi_y + \psi_x^\top F_1(\phi) + \psi_y^\top F_2(\phi) + \varepsilon \psi_y^\top g(t,\phi) \\
       &=: H_0(\phi_x,\phi_y,\psi_x,\psi_y) + \varepsilon H_1(t,\phi_x,\phi_y,\psi_x,\psi_y).
\end{split}
\end{align}
Treating $\psi_x$ and $\psi_y$ as independent variables, and denoting
$\Phi=(\phi,\psi)^{\top}=(\phi_x,\phi_y,\psi_x,\psi_y)^{\top}$ for convenience, Hamilton's equations
take the canonical form
\begin{equation}\label{Ham-Eq}
   \dot{\Phi}=J\nabla_\Phi H_\varepsilon(t,\Phi)
   = \mathcal F(\Phi)+\varepsilon \mathcal G(t,\Phi),
\end{equation}
where $J=
\begin{pmatrix}
O_{n+m} & I_{n+m}\\
-I_{n+m} & O_{n+m}
\end{pmatrix}$ is the Poisson matrix, and
\begin{align}\label{def of F}
    \mathcal F(\phi,\psi):=
\begin{pmatrix}
F_1(\phi)\\[1mm]
F_2(\phi)+ \Sigma(\phi)\psi_y\\[1mm]
-(\partial_{\phi_x} F_1)^{\top} \psi_x
-(\partial_{\phi_x} F_2)^{\top}\psi_y
-\frac12\,\psi_y^\top\big(D_{\phi_x}\Sigma\big)\psi_y\\[1mm]
-(\partial_{\phi_y} F_1)^{\top} \psi_x
-(\partial_{\phi_y} F_2)^{\top}\psi_y
-\frac12\,\psi_y^\top\big(D_{\phi_y}\Sigma\big)\psi_y
\end{pmatrix},
\end{align}
\begin{align}\label{def of G}
\mathcal G(t,\phi,\psi)
:=
\begin{pmatrix}
0\\[1mm]
g(t,\phi)\\[1mm]
-(\partial_{\phi_x} g)^{\top} \psi_y\\[1mm]
-(\partial_{\phi_y} g)^{\top} \psi_y
\end{pmatrix}.
\end{align}

We note that the matrix-derivative identities
\begin{align*}
    (\Sigma \psi_y)^{\top} (D_\phi \Sigma^{-1}) (\Sigma \psi_y) = - \psi_y^{\top} (D_\phi \Sigma) \psi_y,
\end{align*}
follow from differentiating $\Sigma \Sigma^{-1} = I$ and are used to convert terms involving $D_\phi \Sigma^{-1}$ into expressions in $D_\phi \Sigma$. 

Consequently, every MLP $\phi^\ast$ corresponds to a solution $\Phi^*(t)=(\phi^*(t),\psi^*(t))$ of the Hamiltonian system above satisfying the (heteroclinic) boundary conditions:
\begin{align*}
    \begin{cases}
    \lim_{t\to -\infty}\big(\phi^*(t)-z_a^\varepsilon(t)\big)=0,&
\lim_{t\to +\infty}\big(\phi^*(t)-z_b^\varepsilon(t)\big)=0,\\
    \lim_{t\to -\infty} \psi(t)^*= 0, & \lim_{t\to +\infty} \psi^*(t)= 0.
    \end{cases}
\end{align*}

Therefore, the study of the MLP for \eqref{dSDE} is
reduced to the study of heteroclinic connections of the perturbed Hamiltonian
system \eqref{Ham-Eq} in phase space. Indeed, any MLP $\phi^*$ can be lifted to a
Hamiltonian trajectory $\Phi^*(t)=(\phi^*(t),\psi^*(t))$ satisfying the
heteroclinic boundary conditions, while the MLP itself is recovered as the
configuration component of this phase-space connection.



\section{Persistence of Perturbed Heteroclinic Connections}\label{sec4}

In this section, we study the persistence of the uphill heteroclinic connection for the perturbed Hamiltonian system. We first introduce the required structural assumptions and then prove Theorem \ref{thm:persistence-uphill} by geometric Melnikov methods.

\noindent\textbf{Assumption \textbf{(A3)}}
\begin{itemize}
\item[($i$)] \textit{(Regularity and boundedness of the coefficients)} Let \(U\) be a
neighborhood of
 $\Phi_{\mathrm{up}}^0(\mathbb R)\cup\{\Gamma_a,\Gamma_u\}$ in the phase space, and let \(U_\phi\) be its projection onto the configuration space. Assume that $$\Sigma\in C_b^3(U_\phi),\;\mathcal F\in C_b^2(U),\;\mathcal G\in C_b^2(\mathbb R\times U),$$
where $\Sigma(\phi)=\sigma(\phi)\sigma(\phi)^\top$ and $\mathcal F, \mathcal G$ are defined in \eqref{def of F}, \eqref{def of G}, respectively.
 \item[($ii$)] \textit{(Existence of an unperturbed uphill heteroclinic orbit)}
For $\varepsilon=0$, let
\begin{equation}\label{unperturbed-HS}
    \dot\Phi=\mathcal F(\Phi),\qquad \Phi=(\phi,\psi)^{\top},
\end{equation}
be the unperturbed Hamiltonian system associated with \eqref{Ham-Eq}.
We assume that it admits an ``uphill'' heteroclinic orbit
$$
    \Phi_{\mathrm{up}}^0(t)
    =
    \bigl(\phi^*(t),\psi^*(t)\bigr)^{\top},
    \qquad t\in\mathbb R,
$$
with phase-space limits
$$
    \lim_{t\to-\infty}\Phi_{\mathrm{up}}^0(t)=(z_a,0,0)^{\top},
    \qquad
    \lim_{t\to+\infty}\Phi_{\mathrm{up}}^0(t)=(z_u,0,0)^{\top}.
$$
Note that the word ``uphill'' here refers only to the configuration component of this
phase-space orbit, namely
$$
    \lim_{t\to-\infty}\phi^*(t)=z_a,
    \qquad
    \lim_{t\to+\infty}\phi^*(t)=z_u,
$$
with $z_a$ the initial stable state and $z_u$ the intermediate
saddle state.

\item[($iii$)] \textit{(Hyperbolicity and transversality)}
Assume that 
$$
\Gamma_a:=(z_a,0,0)^{\top} \;\;\text{and}\;\; \Gamma_u:=(z_u,0,0)^{\top},
$$
are hyperbolic equilibria of the
    unperturbed Hamiltonian system \eqref{unperturbed-HS}. Furthermore, the unperturbed uphill heteroclinic orbit
    $\Phi_{\mathrm{up}}^0$ is transverse in the zero-energy manifold
    $$
        \mathcal E_0:=\{\Phi:H_0(\Phi)=0\},
    $$
    where $H_0$ is the unperturbed Hamiltonian defined in \eqref{Hf}. More precisely, this means that,
    $$
        T_{\Phi_{\mathrm{up}}^0(0)}W^u(\Gamma_a;\mathcal F)
        +
        T_{\Phi_{\mathrm{up}}^0(0)}W^s(\Gamma_u;\mathcal F)
        =
        T_{\Phi_{\mathrm{up}}^0(0)}\mathcal E_0.
    $$
\item[($iv$)] \textit{(Simple zero of the Melnikov function)} 
Assume that there exists \(\alpha_0\in\mathbb R\) such that
\begin{equation*}
    M_{\mathrm{up}}(\alpha_0)=0,
    \qquad
    M_{\mathrm{up}}'(\alpha_0)\neq0,
\end{equation*}
where the Melnikov
function is defined by
\begin{equation}\label{Mup}
    M_{\mathrm{up}}(\alpha)
    :=
    \int_{-\infty}^{+\infty}
    \left\langle
    \nabla_\Phi H_0(\Phi_{\mathrm{up}}^0(s)),
    \mathcal G(s+\alpha,\Phi_{\mathrm{up}}^0(s))
    \right\rangle ds.
\end{equation}

\end{itemize}
\begin{remark}
Assumption~\textbf{(A3)}($i$)  ensures the invariant manifolds and their perturbation expansions are well defined.
Assumption~\textbf{(A3)}($ii$) supplies the reference orbit for the persistence
analysis. In many applications, such an unperturbed uphill heteroclinic
connection is known to exist. In particular, for gradient systems, and more
generally for Langevin-type systems satisfying a suitable
fluctuation--dissipation structure, the uphill transition path is related to
the time reversal of a deterministic relaxation trajectory \cite{Souza2019,Chao2022,Chao2026}. We impose the existence assumption only on the uphill connection. The reason is
that the downhill part is typically governed by deterministic relaxation: once
the system reaches a neighborhood of the intermediate saddle state $z_u$, it
is then carried by the deterministic flow toward the stable state $z_b$,
without requiring additional noise. In the Hamiltonian formulation, this
corresponds to the zero-momentum branch $\psi\equiv0$. Indeed, on this branch
the Hamiltonian system reduces to 
$
    \dot\phi_x=F_1(\phi),\;
    \dot\phi_y=F_2(\phi),
$ 
which is precisely the deterministic relaxation dynamics. Along such a
deterministic relaxation path, the Freidlin--Wentzell action vanishes. In
section~\ref{sec5}, we shall further show that, for the perturbed system, the action
along the perturbed downhill connection has no linear contribution in
$\varepsilon$.
\end{remark}
\begin{remark}
Assumption~\textbf{(A3)}($iii$) is the key geometric assumption.
The hyperbolicity of $\Gamma_a$ and $\Gamma_u$ guarantees the existence and persistence of their stable and unstable manifolds. The transversality condition in the zero-energy manifold $\mathcal E_0$ excludes degenerate tangencies between $W^u(\Gamma_a;\mathcal F)$ and $W^s(\Gamma_u;\mathcal F)$. Consequently, the splitting of the perturbed manifolds can be reduced to a scalar displacement in the conormal direction to $\mathcal E_0$, and its first-order term is precisely the Melnikov function \eqref{Mup}. In a general Hamiltonian phase space, the
relative splitting of the stable and unstable manifolds may be higher-dimensional,
and the corresponding Melnikov object is typically vector-valued. Melnikov functions for heteroclinic manifolds have been widely studied in the periodic case \cite{Yagasaki2025}, and have been extended to aperiodically
time-dependent settings \cite{Wiggins1998} and more general time-dependent
heteroclinic settings \cite{Priyankara2022}. In addition, see \cite{GideaLlave2018} for global Melnikov theory on
homoclinic manifolds with general time-dependent perturbations. Finally, Assumption~\textbf{(A3)}($iv$) allows one to use the implicit function theorem for the first-order splitting equation and to choose a suitable phase shift $\alpha_\varepsilon$ for which the perturbed stable and unstable manifolds intersect.
\end{remark}

\begin{theorem}[Persistence of the perturbed uphill heteroclinic connection]
\label{thm:persistence-uphill}
Assume that Assumptions~\textbf{(A1)-(A3)} hold, then there exists \(\varepsilon_0>0\) such
that, for every \(|\varepsilon|<\varepsilon_0\), there exist hyperbolic
trajectories \(\Gamma_a^\varepsilon(t)\) and \(\Gamma_u^\varepsilon(t)\) of
the perturbed
Hamiltonian system  \eqref{Ham-Eq} satisfying
\begin{equation}\label{Gamma}
    \sup_{t\in\mathbb R}
    |\Gamma_a^\varepsilon(t)-\Gamma_a|
    +
    \sup_{t\in\mathbb R}
    |\Gamma_u^\varepsilon(t)-\Gamma_u|
    =
    O(\varepsilon).
\end{equation}
Moreover, there exists $\alpha_\varepsilon=\alpha_0+O(\varepsilon)$
 and a solution
\[
    \Phi_{\mathrm{up}}^\varepsilon(t;\alpha_\varepsilon)
    =
    \bigl(
    \phi_{\mathrm{up}}^\varepsilon(t;\alpha_\varepsilon),
    \psi_{\mathrm{up}}^\varepsilon(t;\alpha_\varepsilon)
    \bigr)^{\top},
\]
of \eqref{Ham-Eq} such that
\begin{equation}\label{pphi}
    \lim_{t\to-\infty}
    |\Phi_{\mathrm{up}}^\varepsilon(t;\alpha_\varepsilon)
    -
    \Gamma_a^\varepsilon(t)|
    =
    0,
    \qquad
    \lim_{t\to+\infty}
    |\Phi_{\mathrm{up}}^\varepsilon(t;\alpha_\varepsilon)
    -
    \Gamma_u^\varepsilon(t)|
    =
    0.
\end{equation}
In particular, its configuration component gives a perturbed uphill connection
from a neighborhood of \(z_a\) to a neighborhood of \(z_u\).
\end{theorem}
\begin{proof}
We divide the proof into three steps.

\textbf{Step 1. Persistence of the endpoint trajectories.}
By Assumption~\textbf{(A3)} ($iii$), \(\Gamma_a\) and \(\Gamma_u\) are
hyperbolic equilibria of the unperturbed Hamiltonian system
\[
    \dot\Phi=\mathcal F(\Phi).
\]
Since \(\mathcal G\) is uniformly bounded and sufficiently smooth, the standard
persistence theorem for hyperbolic trajectories under small nonautonomous
perturbations implies that, for \(|\varepsilon|\) sufficiently small, the
perturbed system \eqref{Ham-Eq} admits hyperbolic trajectories
\(\Gamma_a^\varepsilon(t)\) and \(\Gamma_u^\varepsilon(t)\) satisfying \eqref{Gamma}.
Their stable and unstable manifolds also persist and depend smoothly on
\(\varepsilon\).

\textbf{Step 2. Construction of the Melnikov splitting function.}
Let \(\alpha\in\mathbb R\) be the time-shift parameter. By Assumption~\textbf{(A3)}($ii$), there exists a reference
orbit \(\Phi_{\mathrm{up}}^0(s)\), and the corresponding physical time is
\(s+\alpha\). Thus the perturbed equation \eqref{Ham-Eq} can be written as
\begin{equation}\label{shifted-eq}
    \frac{d\Phi}{ds}
    =
    \mathcal F(\Phi)
    +
    \varepsilon \mathcal G(s+\alpha,\Phi).
\end{equation}
By Assumption~\textbf{(A3)}($iii$), the unperturbed stable and unstable
manifolds intersect transversely in the zero-energy manifold
$\mathcal E_0=\{H_0=0\}$. Hence the only non-tangential splitting direction
is the one normal to $\mathcal E_0$. This direction is represented by the
conormal covector $v(s)=\nabla_\Phi H_0(\Phi_{\mathrm{up}}^0(s))$.
Since $H_0$ is a first integral of the unperturbed Hamiltonian flow, $v(s)$
satisfies the adjoint variational equation
\[
    \dot v(s)
    =
    -D\mathcal F(\Phi_{\mathrm{up}}^0(s))^{\top}v(s).
\]
Thus $v(s)$ spans the one-dimensional conormal direction used to measure the
leading-order distance between the perturbed stable and unstable manifolds.

Here we denote $\Phi^{0,u}_{up}(s)$ (resp. $\Phi^{0,s}_{up}(s)$), the point of $\Phi^0_{up}(s)$ on the unstable manifold  $W^u(\Gamma_a)$ (resp. on the stable manifold $W^s(\Gamma_u)$). The persistence results of an invariant manifold associated with hyperbolic fixed points \cite{Yi1993} indicate the presence of a perturbed unstable (stable) manifold, $W^u(\Gamma_a^{\varepsilon}(s+\alpha))$ (resp. $W^s(\Gamma_u^{\varepsilon}(s+\alpha))$) is $\mathcal O(\varepsilon)$-close to $W^u(\Gamma_a)$ (resp. $W^s(\Gamma_u)$) at finite times $s+\alpha$. There is a point $\Phi_{up}^{\varepsilon,u}(s;\alpha)$ (resp. $\Phi_{up}^{\varepsilon,s}(s;\alpha)$) on the perturbed manifold which is $\mathcal O(\varepsilon)$-close to $\Phi_{up}^{0,u}(s)$ (resp. $\Phi_{up}^{0,s}(s)$). Our aim is to quantify the distance $D^u(s;\alpha,\varepsilon)$ obtained by projecting the vector $\Phi_{\mathrm{up}}^{\varepsilon,u}(s;\alpha)-\Phi_{\mathrm{up}}^{0,u}(s)$ onto the conormal covector to $W^u(\Gamma_a)$ at the point $\Phi_{\mathrm{up}}^{0,u}(s)$. By doing so, we can describe the location of the perturbed manifold $W^u(\Gamma_a^{\varepsilon})$ parametrized by time $\alpha$ to the leading-order in $\varepsilon$. The displacement of $\Phi_{\mathrm{up}}^{\varepsilon,u}(s;\alpha)$ from the
unperturbed orbit $\Phi_{\mathrm{up}}^{0,u}(s)$ in the conormal direction
$v(s)$ is measured by
\begin{equation}\label{Du}
D^u(s;\alpha,\varepsilon)
=
\left\langle
v(s),
\Phi_{\mathrm{up}}^{\varepsilon,u}(s;\alpha)
-
\Phi_{\mathrm{up}}^{0,u}(s)
\right\rangle.
\end{equation}

Recall the unperturbed equation
   $\dot\Phi_{\mathrm{up}}^{0,u}(s)
    =
    \mathcal F(\Phi_{\mathrm{up}}^{0,u}(s))$,
and the adjoint equation for \(v(s)\), by \eqref{shifted-eq}, using Taylor expansions for \(\mathcal F\) and
\(\mathcal G\) and taking  the derivative of \eqref{Du} with respect to $s$ yields 
\begin{align*}
    \frac{d}{ds}D^u(s;\alpha,\varepsilon)
    &=
    \left\langle
    \dot v(s),
    \Phi_{\mathrm{up}}^{\varepsilon,u}(s;\alpha)
    -
    \Phi_{\mathrm{up}}^{0,u}(s)
    \right\rangle \notag\\
    &\quad+
    \left\langle
    v(s),
    \frac{d}{ds}\Phi_{\mathrm{up}}^{\varepsilon,u}(s;\alpha)
    -
    \dot\Phi_{\mathrm{up}}^{0,u}(s)
    \right\rangle \notag\\
    &=
    \left\langle
    -D\mathcal F(\Phi_{\mathrm{up}}^{0,u}(s))^\top v(s),
    \Phi_{\mathrm{up}}^{\varepsilon,u}(s;\alpha)
    -
    \Phi_{\mathrm{up}}^{0,u}(s)
    \right\rangle \notag\\
    &\quad+
    \left\langle
    v(s),
    \mathcal F(\Phi_{\mathrm{up}}^{\varepsilon,u}(s;\alpha))
    -
    \mathcal F(\Phi_{\mathrm{up}}^{0,u}(s))
    \right\rangle \notag\\
    &\quad+
    \varepsilon
    \left\langle
    v(s),
    \mathcal G(s+\alpha,\Phi_{\mathrm{up}}^{\varepsilon,u}(s;\alpha))
    \right\rangle  \notag\\
    &=\varepsilon
    \left\langle
    v(s),
    \mathcal G(s+\alpha,\Phi_{\mathrm{up}}^{0,u}(s))
    \right\rangle
    +
    O(\varepsilon^2),
\end{align*}
where the boundedness of the derivatives of \(\mathcal F\) and
\(\mathcal G\) is also used.
Integrating from \(-\infty\) to \(s\), we obtain
\begin{equation*}
    D^u(s;\alpha,\varepsilon)
    =
    \varepsilon
    \int_{-\infty}^{s}
    \left\langle
    v(\tilde s),
    \mathcal G(\tilde s+\alpha,\Phi_{\mathrm{up}}^{0,u}(\tilde s))
    \right\rangle d\tilde s
    +
    R^u(s;\alpha,\varepsilon),
\end{equation*}
where $R$ is the remainder term.
Similarly, the displacement of the stable manifold  $D^s(s;\alpha,\varepsilon)$ can be expanded in $\varepsilon$ in the expression $$D^s(s;\alpha,\varepsilon)=-\varepsilon \int^{+\infty}_{s} \left\langle
    v(\tilde s),
    \mathcal G(\tilde s+\alpha,\Phi_{\mathrm{up}}^{0,s}(\tilde s))
    \right\rangle d\tilde s+R^s(s;\alpha,\varepsilon).$$
Later we shall quantify the displacement between the perturbed unstable and stable manifolds, in the direction conormal to $W(\Phi_{up}^0)$: 
$$D(s;\alpha,\varepsilon)=\langle v(s), \Phi_{up}^{\varepsilon,u}(s;\alpha)-\Phi_{up}^{\varepsilon,s}(s;\alpha)\rangle.$$
Since $\Phi^0_{up}=\Phi^{0,u}_{up}=\Phi^{0,s}_{up}$ in this instance, we have $D(s;\alpha,\varepsilon)=D^u(s;\alpha,\varepsilon)-D^s(s;\alpha,\varepsilon)$.  Finally, we can obtain that $$D(s;\alpha,\varepsilon)=\varepsilon M_{up}(\alpha)+R(s;\alpha,\varepsilon),$$ where the Melnikov function is the improper integral \eqref{Mup} with $s$ instead of $\tilde s$.

We next verify that the Melnikov integral is well-defined. Since
\(\Gamma_a\) and \(\Gamma_u\) are hyperbolic and
\(\Phi_{\mathrm{up}}^0\) is a heteroclinic orbit, there exist constants
\(C>0\) and \(\lambda>0\) such that
\begin{equation*}
    |\dot\Phi_{\mathrm{up}}^0(s)|
    \le
    C e^{-\lambda |s|},
    \qquad s\in\mathbb R.
\end{equation*}
Using the equality
\[
    \dot\Phi_{\mathrm{up}}^0(s)
    =
    J\nabla_\Phi H_0(\Phi_{\mathrm{up}}^0(s)),
\]
and the nonsingularity of \(J\), we have
\begin{equation*}
    |v(s)|
    =
    |\nabla_\Phi H_0(\Phi_{\mathrm{up}}^0(s))|
    \le
    C e^{-\lambda |s|}.
\end{equation*}
Since \(\mathcal G\) is uniformly bounded on \(\mathbb R\times U\), $M_{up}(\alpha)$ is absolutely convergent. Moreover,
\begin{align*}
    |M_{\mathrm{up}}(\alpha)|
    &\le
    C
    \int_{-\infty}^{+\infty}e^{-\lambda |s|}\,ds
    <
    \infty .
\end{align*}
Besides, there exists a constant $\eta>0$, such that the remainder term satisfies a weighted estimate
\[
    |R(s;\alpha,\varepsilon)|
    \le C\varepsilon^2 \int_{-\infty}^{+\infty}e^{-\eta |\tilde s|}d\tilde s.
\]
Therefore the remainder term is integrable on \(\mathbb R\), and the integrated error is of order \(O(\varepsilon^2)\).

\textbf{Step 3. Persistence of the uphill heteroclinic connection.}
Fix an arbitrary reference section $s=s_0$ along the unperturbed heteroclinic
orbit $\Phi_{\mathrm{up}}^0$. We define the scalar splitting function on this
section by
\[
    D_{s_0}(\alpha,\varepsilon)
    :=
    \left\langle
    v(s_0),
    \Phi_{\mathrm{up}}^{\varepsilon,u}(s_0;\alpha)
    -
    \Phi_{\mathrm{up}}^{\varepsilon,s}(s_0;\alpha)
    \right\rangle.
\]
By Assumption~\textbf{(A3)} ($iv$), there exists a constant \(\alpha_0\in\mathbb R\)
such that
\[
    M_{\mathrm{up}}(\alpha_0)=0,
    \qquad
    M_{\mathrm{up}}'(\alpha_0)\neq0.
\]
Therefore, by the implicit function theorem, there exists a constant
    $\alpha_\varepsilon=\alpha_0+O(\varepsilon)$
such that $D_{s_0}(\alpha_\varepsilon,\varepsilon)=0.$
Since the unperturbed stable and unstable manifolds intersect transversely
inside the zero-energy manifold $\mathcal E_0$, the tangential displacement
does not prevent their intersection. The essential obstruction is the
displacement in the conormal direction.  Hence, it is enough to measure the
splitting by \(D_{s_0}\).
Thus the perturbed unstable manifold of \(\Gamma_a^\varepsilon\) and the
perturbed stable manifold of \(\Gamma_u^\varepsilon\) intersect near the
unperturbed connection, yielding \eqref{pphi}.
The persistence of the perturbed uphill heteroclinic connection follows.
\end{proof}

\section{Transition Rates: Periodic Case}\label{sec5}
In this section, we consider the periodic case of \eqref{dSDE}, i.e., there exists a constant $\tau> 0$ such that $g(t+\tau,\cdot,\cdot)=g(t,\cdot,\cdot)$. To analyze the corresponding optimal transition rate, we first formulate the associated variational problem over one time period, and then identify the leading-order contribution to the action functional.

\subsection{Suspended Flow and  Heteroclinic Connections}

We shall understand the dynamics of nonautonomous systems \eqref{Ham-Eq} by the associated Poincar\'e map. Upon introducing a phase $\theta\in\mathcal S^1 (\mathcal S^1=\mathbb R/{\tau\mathbb Z})$, the system \eqref{Ham-Eq} becomes autonomous on the extended phase space $\mathbb R^{2n+2m}\times\mathcal S^1$:
\begin{align}\label{sus}
	\left\{\begin{array}{l}
		\dot\theta = 1,\\
		\dot{\phi}_x = F_1(\phi),\\
\dot{\phi}_y = \Sigma \psi_y + F_2(\phi) + \varepsilon g(\theta,\phi),\\
\dot{\psi}_x = -\big(\partial_{\phi_x} (F_2(\phi)+\varepsilon g(\theta,\phi))\big)^{\top}\psi_y
    - (\partial_{\phi_x} F_1(\phi))^{\top} \psi_x
    + \frac{1}{2} (\Sigma(\phi)\psi_y)^{\top} (\partial_{\phi_x} \Sigma^{-1}(\phi)) \Sigma(\phi)\psi_y,\\
\dot{\psi}_y = -\big(\partial_{\phi_y} (F_2(\phi)+\varepsilon g(\theta,\phi))\big)^{\top} \psi_y
    - (\partial_{\phi_y} F_1(\phi))^{\top} \psi_x
    + \frac{1}{2} (\Sigma(\phi)\psi_y)^{\top} (\partial_{\phi_y} \Sigma^{-1}(\phi)) \Sigma(\phi)\psi_y, 
	\end{array}\right.
\end{align}
where $(\phi_x, \phi_y, \psi_x, \psi_y, \theta) \in \mathbb{R}^{n} \times \mathbb{R}^{m} \times \mathbb{R}^{n} \times \mathbb{R}^{m} \times \mathcal{S}^1$. Since the $\theta$ equtation is independent of $\phi,\psi$, this is a skew-product system. For sufficiently small $\varepsilon$, \eqref{sus} admits a Poincar\'e map
$$
P_\varepsilon^{t_0} : \Sigma_{t_0} \to \Sigma_{t_0}, \quad
\Sigma_{t_0} = \{ (\phi_x,\phi_y,\psi_x,\psi_y,\theta)^{\top} \mid \theta = t_0 \in [0, \tau] \},
$$
which is a global cross-section at time $t_0$ of the suspended flow. For the case of $\varepsilon=0$, the Poincar\'e map $P_{\varepsilon}^{t_0}$ has hyperbolic fixed points at $\Gamma_a,\Gamma_u,\Gamma_b$, where $\Gamma_a:=(z_a,\mathbf 0), \Gamma_u:=(z_u,\mathbf 0), \Gamma_b:=(z_b,\mathbf 0)$. Under Assumption (A3), the suspended system \eqref{sus} has circular orbits $$
\gamma_a = \Gamma_a \times \mathcal{S}^1, \quad
\gamma_b = \Gamma_b \times \mathcal{S}^1, \quad
\gamma_u = \Gamma_u \times \mathcal{S}^1,
$$ with unstable and stable manifolds $W^u(\gamma_a)$ (resp., $W^u(\gamma_u)$) and $W^s(\gamma_u)$ (resp., $W^s(\gamma_b)$) coincide  and form a $2$-dimensional “cylinder” $\Phi_{up}^0 \times S^1$ (resp., $\Phi_{down}^0\times S^1$), called a heteroclinic manifold, where $\Phi_{up}^0$ (resp., $\Phi_{down}^0$) is a heteroclinic connection between $\Gamma_a$ and $\Gamma_u$ (resp., $\Gamma_u$ and $\Gamma_b$). 

It is a natural question to ask: how to calculate variations in the transition rate of the heteroclinic orbit $\Phi_{up}^0$ and $\Phi_{down}^0$ under small periodic perturbations?

\subsection{Characterization of Transition Rate Variations}
We emphasize that the parameter $t_0$ used in this section is the periodic
counterpart of the time-shift parameter $\alpha$ introduced in Section~\ref{sec4}. Hence, a Melnikov zero $\alpha_0$ corresponds
to a critical phase $t_0=\alpha_0 \mod \tau$.

Section~\ref{sec4} indicates that the existence of a solution $\Phi_{up}^{\varepsilon}(t;t_0)$ (resp., $\Phi_{down}^{\varepsilon}(t;t_1)$) of \eqref{Ham-Eq}, which defines a heteroclinic connection between the unstable manifold of $\Gamma_a^{\varepsilon}$ (resp., $\Gamma_u^{\varepsilon}$) and the stable manifold of $\Gamma_u^{\varepsilon}$ (resp.,$\Gamma_b^{\varepsilon}$) on the section $\Sigma_{t_0}$ (resp.,$\Sigma_{t_1}$) of phase space. Therefore, the first and second components of this solution naturally give rise to the existence of a solution of \eqref{EL} defining a heteroclinic connection $\phi_{up}^{\varepsilon}(t;t_0),$ (resp., $\phi_{down}^{\varepsilon}(t;t_1)$) from $z_a^\varepsilon(t)$ (resp., $z_u^\varepsilon(t)$) to $z_u^\varepsilon(t)$ (resp., $z_b^\varepsilon(t)$) in the configuration space. Furthermore, for each given $t_0$ (resp., $t_1$), this heteroclinic connection is regarded as an MLP of the original stochastic system~\eqref{dSDE}. 

Different from the autonomous case, the concatenation of the perturbed heteroclinic orbits $\phi_{up}^{\varepsilon}(t;t_0)$ and $\phi_{down}^{\varepsilon}(t;t_1)$ yields different critical points of the action functional for the system \eqref{dSDE}. Therefore, we optimize over $t_0$ and $t_1$ to obtain an MLP that is independent of these parameters and determines the optimal transition rate for system \eqref{dSDE}. Consequently, \eqref{FW-deg} can be formally rewritten as follows~\cite{Tao2019,Chao2022}:
\begin{equation}\label{S}
\mathbf S^\varepsilon = \min_{t_0}\{S^\varepsilon[\phi_{up}^{\varepsilon}(t;t_0)]\} + \min_{t_1}\{S^\varepsilon[\phi_{down}^{\varepsilon}(t;t_1)]\}.
\end{equation}
Since $\phi_{down}^{\varepsilon}(t;t_1)$ is always moving along the perturbed downhill heteroclinic orbits, the first-order term in the Taylor expansion of $S^\varepsilon[\phi_{down}^{\varepsilon}(t;t_1)]$ with respect to $\varepsilon$ is zero. 
Therefore, we only need to focus on optimizing $S^\varepsilon[\phi_{up}^{\varepsilon}(t;t_0)]$ over $t_0$. That is, $$\mathbf S^\varepsilon = \min_{t_0} S^\varepsilon[\phi_{up}^{\varepsilon}(t;t_0)].$$

We now give the following main theorem of this section.
\begin{theorem}\label{the1}
Consider the degenerate non-autonomous SDE \eqref{dSDE} with a $\tau$-periodic perturbation $g(t,x,y)$. Under Assumptions \textbf{(A1)--(A3)}, for sufficiently small $\varepsilon$, there exists  an uphill heteroclinic connection from stable periodic orbits $z_a^\varepsilon(t)$ to unstable periodic orbits $z_u^\varepsilon(t)$ and a deterministic downhill connection from $z_u^\varepsilon(t)$ to stable periodic orbits $z_b^\varepsilon(t)$. In addition, the escape rate $\mathcal R_\varepsilon$ from $z_a^{\varepsilon}(t)$ to $z_b^{\varepsilon}(t)$ is asymptotically equivalent to $exp(-\mathbf S^{\varepsilon}/\mu)$, where $\mathbf S^\varepsilon = S_0+\varepsilon\delta S_e+O(\varepsilon^2)$, and $\delta S_e$ is given by
\begin{align}
   \delta S_e& = \min_{t_0} \delta S(t_0),\nonumber 
\end{align}
where
\begin{equation}
\delta S(t_0) = -\int_{-\infty}^{+\infty}\psi_y^*(t)^{\top}g(t+t_0,\phi^*(t))dt.\nonumber
\end{equation}
\end{theorem}
\begin{proof} 

Keep in mind that the relationship among the Euler-Lagrange problem \eqref{EL}, the Hamiltonian problem \eqref{Ham-Eq} and the local minimizers (i.e., MLPs) of action functional $S^\varepsilon[\phi(t)]$ are shown in section~\ref{sec:main1} and~\ref{sec3}. Combing figure~\ref{thm:persistence-uphill}, we can easily obtain a heteroclinic connection from $z_a^{\varepsilon}$ to $z_b^{\varepsilon}$ via $z_u^{\varepsilon}$.

\textbf{Step 1: Reformulation of the Freidlin--Wentzell action functional}
First, we reformulate the action functional $S^{\varepsilon}[\cdot]$ in \eqref{MLP} into an equivalent form (i.e., $A^{\varepsilon}[\cdot]$). 
We introduce a Hamiltonian action
\begin{align}\begin{split}\label{HA}
A^\varepsilon [\Phi] &= \int_{\Phi} \psi_x^{\top}d\phi_x+\psi_y^{\top}d\phi_y-H_{\varepsilon}(t,\phi,\psi)dt,
\end{split}\end{align}
where $\Phi= \{ (\phi_x,\phi_y,\psi_x,\psi_y)^{\top}, -\infty < t < \infty \}$ is a path in phase space, and for the expression of $H_\varepsilon$, see \eqref{Hf}. Stationary curves of the action~\eqref{HA} are Hamiltonian trajectories~\eqref{Ham-Eq} and $\phi_x,\phi_y,\psi_x,\psi_y$ in \eqref{HA} are mutually independent. By exploiting the equivalence of Lagrangian mechanics and Hamiltonian mechanics, we have 
$$S^{\varepsilon}[\phi]=A^{\varepsilon}[\Phi],$$ along the instanton solutions of \eqref{EL} or \eqref{Ham-Eq}. More precisely,
\begin{equation}\label{sa}S^{\varepsilon}[\phi^{\varepsilon}_{up}(t;t_0)]=A^{\varepsilon}[\Phi_{up}^{\varepsilon}(t;t_0)].\end{equation}
This identity follows directly from the Legendre duality between the Lagrangian and Hamiltonian formulations along the corresponding instanton trajectories.
For more details on the proof, see~\cite{Meiss2007}. 

\textbf{Step 2: Minimum of the
FW action functional for system \eqref{dSDE} with $\varepsilon=0$}

When $\varepsilon=0$, we focus on the transition from $z_a$ to $z_u$. Note that in this case the heteroclinic orbit is $\Phi_{up}^0=(\phi_x^*,\phi_y^*,\psi_x^*,\psi_y^*)$, we find that 
\begin{equation*}\begin{aligned}
A_{0}[\Phi] & =\int\left(
(\psi_y^*)^{\top}\bigl(\dot{\phi}^*_y - F_2(\phi^*)\bigr)
-\frac12\,(\psi_y^*)^{\top}\Sigma\psi_y^*
\right)dt\\&=\frac12\int\|\Sigma\psi_y^*\|_{\Sigma^{-1}}^2dt,
\end{aligned}\end{equation*}
and
\begin{equation*}\begin{aligned}
S_{0}[\phi]=\frac{1}{2}\int_{-\infty}^{+\infty}\|\dot{\phi^*_{y}}-F_{2}\|_{\Sigma^{-1}}^{2}dt=\frac12\int\|\Sigma\psi_y^*\|_{\Sigma^{-1}}^2dt,\end{aligned}
\end{equation*}
under the boundary conditions: $\psi^*(-\infty)=\mathbf 0,\;\psi^*(+\infty)=\mathbf 0$.
One can easily note that $S_0=A_0$. The same argument applies to the transition from $z_u^{\varepsilon}(t)$ to $z_b^{\varepsilon}(t)$, for which the corresponding action $S_0$ is zero when $\varepsilon = 0$.

\textbf{Step 3: The relation of the minimizers of $S^{\varepsilon}[\cdot]$ and $S[\cdot]$.}

In the following, we will use the reformulation \eqref{sa} to study the relationship between $S^{\varepsilon}[\cdot]$ and $S[\cdot]$, which is equivalent to dealing with the relationship between $A^{\varepsilon}[\cdot]$ and $A[\cdot]$.

 In light of \textbf{Step 1}, we now perform a linear-theory calculation of the action $S^{\varepsilon}[\cdot]$ inspired by~\cite{Assaf2008} and approximate the rate of the metastable transition. Since Assumption \textbf{(A2)} implies that $\varepsilon\ll1$, the term $\varepsilon H_1(t,\phi_x,\phi_y,\psi_x,\psi_y)$ in \eqref{Hf} can be treated perturbatively. Applying Taylor expansions to the perturbed MLP yields
 $$\Phi_{up}^{\varepsilon}(t;t_0)=\Phi_{up}^{0}(t-t_0)+\varepsilon\bar\Phi_{up}(t;t_0)+O(\varepsilon^2),$$ where $\bar\Phi_{up}(t;t_0)=(\bar\phi_x,\bar\phi_y,\bar\psi_x,\bar\psi_y)^{\top}(t;t_0).$
Since $H^\varepsilon(t,\Phi)=H_0(\Phi)+\varepsilon H_1(t,\Phi)$, we have
\[
H^\varepsilon(t,\Phi^\varepsilon_{up})
=
H_0(\Phi^0_{up})
+\varepsilon DH_0(\Phi^0_{up})\bar\Phi_{up}
+\varepsilon H_1(t,\Phi^0_{up})
+O(\varepsilon^2).
\]
Applying a Taylor expansion to the action $A^{\varepsilon}[\cdot]$ given in \eqref{HA}, we have that
\begin{align*}
    A^{\varepsilon}[\Phi_{up}^{\varepsilon}(t;t_0)]&=\int_{-\infty}^{+\infty}\bigg((\psi_x^{*}+\varepsilon\bar\psi_x)^{\top}(\dot\phi_x^{*}+\varepsilon\dot{\bar\phi}_x)+(\psi_y^{*}+\varepsilon\bar\psi_y)^{\top}(\dot\phi_y^{*}+\varepsilon\dot{\bar\phi}_y)\\&\quad-\big(H_0(\Phi_{up}^0(t-t_0))+\varepsilon\bar\phi_x^{\top}\frac{\partial H_0}{\partial\phi_x}+\varepsilon\bar\phi_y^{\top}\frac{\partial H_0}{\partial\phi_y}+\varepsilon\bar\psi_x^{\top}\frac{\partial H_0}{\partial\psi_x}\nonumber\\&\quad+\varepsilon\bar\psi_y^{\top}\frac{\partial H_0}{\partial\psi_y}\big)+\varepsilon H_1(\Phi_{up}^0(t-t_0))\bigg)dt+O(\varepsilon^2)\\&=A_0+\varepsilon [\psi_x^{\top}\bar\phi_x+\psi_y^{\top}\bar\phi_y]|_{-\infty}^{+\infty}-\varepsilon\int_{-\infty}^{+\infty}H_1(\Phi_{up}^0(t-t_0))dt+O(\varepsilon^2).
\end{align*} 
Using the boundary conditions, we have $\psi_x^{\top}\bar\phi_x+\psi_y^{\top}\bar\phi_y|_{-\infty}^{+\infty}=0$. Thus, we finally conclude that $S^{\varepsilon}=S_0+\varepsilon\delta S(t_0)+O(\varepsilon^2)$ with
\begin{align}\begin{split}\label{delta}
\delta S(t_{0}) &=-\int_{-\infty}^{+\infty}H_{1}(t,\phi_{x}^*(t-t_0),\phi_{y}^*(t-t_0),\psi_{x}^*(t-t_0),\psi_{y}^*(t-t_0))dt \\
 & =-\int_{-\infty}^{+\infty}\psi_y^*(t-t_{0})^{\top}g(t,\phi_{x}^*(t-t_{0}),\phi_{y}^*(t-t_{0}))dt\\&=-\int_{-\infty}^{+\infty}\psi_y^*(t)^{\top}g(t+t_0,\phi^*(t))dt.\end{split}
\end{align} 
Since $\psi_y(t;t_0)\equiv0$ along the downhill heteroclinic orbits, a similar procedure can be used to understand the transition from  $z_u^\varepsilon(t)$ to $z_b^\varepsilon(t)$, and $\delta S(t_0)$ given in \eqref{delta} will vanish for this case. Together with \textbf{Step 2} and the fact that $S_0=0$, we verify that $S^\varepsilon[\phi_{down}^{\varepsilon}(t;t_1)]$ in \eqref{S} is zero to the first order of $\varepsilon$. This finishes the proof.
\end{proof}
\begin{remark}
    To find the optimal first-order correction to the minimum action's value, we have to minimize $\delta S(t_0)$ with respect to $t_0$, which thus yields the equation for the optimal $t_0$. From~\eqref{delta}, we have
\begin{align*}
\frac{d\delta S(t_0)}{dt_0}
&=
-\int_{-\infty}^{+\infty}
\partial_t H_1\bigl(t+t_0,\Phi_{\mathrm{up}}^0(t)\bigr)\,dt .
\end{align*}
Along the unperturbed Hamiltonian orbit $\Phi_{\mathrm{up}}^0(t)$, the chain rule gives
\begin{align*}
\frac{d}{dt}
H_1\bigl(t+t_0,\Phi_{\mathrm{up}}^0(t)\bigr)
&=
\partial_t H_1\bigl(t+t_0,\Phi_{\mathrm{up}}^0(t)\bigr)
+
\{H_1,H_0\}
\bigl(t+t_0,\Phi_{\mathrm{up}}^0(t)\bigr).
\end{align*}
Since
$H_1\bigl(t+t_0,\Phi_{\mathrm{up}}^0(t)\bigr)
\to0,
\;t\to\pm\infty,$ integrating the above identity over $\mathbb R$ yields
\begin{align*}
\int_{-\infty}^{+\infty}
\partial_t H_1\bigl(t+t_0,\Phi_{\mathrm{up}}^0(t)\bigr)\,dt
&=
-\int_{-\infty}^{+\infty}
\{H_1,H_0\}
\bigl(t+t_0,\Phi_{\mathrm{up}}^0(t)\bigr)\,dt .
\end{align*}
Hence,
\begin{align*}
\frac{d\delta S(t_0)}{dt_0}
&=
\int_{-\infty}^{+\infty}
\{H_1,H_0\}
\bigl(t+t_0,\Phi_{\mathrm{up}}^0(t)\bigr)\,dt .
\end{align*}
We now relate this condition to the Melnikov function. Since $\mathcal G(t,\Phi)
=
J\nabla_\Phi H_1(t,\Phi)$, we have
\begin{align*}
M_{\mathrm{up}}(t_0)
&=
\int_{-\infty}^{+\infty}
\left\langle
\nabla_\Phi H_0(\Phi_{\mathrm{up}}^0(t)),
\mathcal G(t+t_0,\Phi_{\mathrm{up}}^0(t))
\right\rangle dt \notag\\
&=
\int_{-\infty}^{+\infty}
\nabla_\Phi H_0(\Phi_{\mathrm{up}}^0(t))^{\top}
J\nabla_\Phi H_1(t+t_0,\Phi_{\mathrm{up}}^0(t))\,dt \notag\\
&=
-\int_{-\infty}^{+\infty}
\{H_1,H_0\}
\bigl(t+t_0,\Phi_{\mathrm{up}}^0(t)\bigr)\,dt .
\end{align*}
Therefore, we have the following expression, 
\begin{align*}
\frac{dS^{\varepsilon}(t_0)}{dt_0}=\varepsilon\frac{d\delta S(t_0)}{dt_0}+O(\varepsilon^2)=\varepsilon\int_{-\infty}^{+\infty}\{H_1,H_0\}dt+O(\varepsilon^2)=-\varepsilon M_{\mathrm{up}}(t_0)+O(\varepsilon^2), \end{align*} which means that $M_{up}(t_0)$ has simple zeros yields a sufficient condition for the existence of perturbed heteroclinc connection, and these zeros are critical points of the function $S^{\varepsilon}(t_0)$. 
\end{remark}

\section{Illustrative Examples}\label{sec6}

Gradient systems and Langevin systems are natural examples, for which the uphill transition path is governed by time-reversal trajectories; see \cite{Souza2019,Chao2022,Chao2026}. In this section, we present two non-gradient examples that cannot be described by such time-reversal arguments.

\paragraph{Example A: A Degenerate 2D Orthogonal Stochastic System}

We now present a simple illustrative example for our theoretical results. Let us first consider the autonomous orthogonal system:
\begin{align*}
    dz=(-\nabla V(z)+b(z))dt+\begin{pmatrix}
0 \\
\sqrt\mu
\end{pmatrix} dB_t,\;\;z=(x,y)\in \mathbb R^{2},
\end{align*} 
with $V(x,y) = \frac{(1-y^2)^2}{4}$ and $b(x,y) =(
1 - y^2 - x,
0)^\top$, where the functions satisfy
$$
\nabla V(z)\cdot b(z)=0
$$
for all $z\in\mathbb R^{2}$.
Then let $G(z,t)=(0,g(z,t))^{\top}$, and $g(z,t)=h(z)cos(\omega t+\vartheta)$, where $h$ is a smooth function, $\theta$ is a constant phase and $\omega\ne0$. Therefore non-autonomous orthogonal system reduces to:\begin{equation}\label{orthogonal_1} 
\begin{cases}
dx = (1 - y^2 - x)dt,\\[1mm]
dy = (y - y^3 +\varepsilon h(x,y)cos(\omega t+\vartheta))dt+\sqrt{\mu}\, dB_t.
\end{cases}
\end{equation} 
When $\varepsilon=\mu=0$, there are 3 fixed points: 
$z_a=(0,-1)^{\top},\; z_u=(1,0)^{\top},\; z_b=(0,1)^{\top},$ 
with $z_a$ and $z_b$ being sinks and $z_u$ being a saddle.

When $\varepsilon=0,$ \cite{Tao2018} shows that there exists a heteroclinc orbit satisfying the uphill equation $dx = (1 - y^2 - x)dt,\;
dy = (-y + y^3)dt$ with boundary condition $(x,y)(-\infty)=z_a,\;(x,y)(+\infty)=z_u$  and the downhill connection is naturally generated by the deterministic flow.
Figure~\ref{Fig1}(a) illustrates a noise-induced transition from $z_a$ to $z_b$ in the absence of periodic forcing ($\varepsilon=0$). Most sample paths remain trapped in a neighborhood of the stable equilibrium $z_a$ for long periods of time. Under the influence of small noise, however, some paths escape from the basin of attraction of $z_a$, pass through a neighborhood of the saddle point $z_u$, and eventually approach the stable equilibrium $z_b$. The heteroclinic connection through $z_u$ therefore provides the deterministic skeleton of the transition and identifies the most probable transition route from $z_a$ to $z_b$.

\begin{figure}[htbp]
\hspace{-1cm}
\begin{minipage}[t]{0.4\linewidth}
\centering
\includegraphics[height=5.8cm]{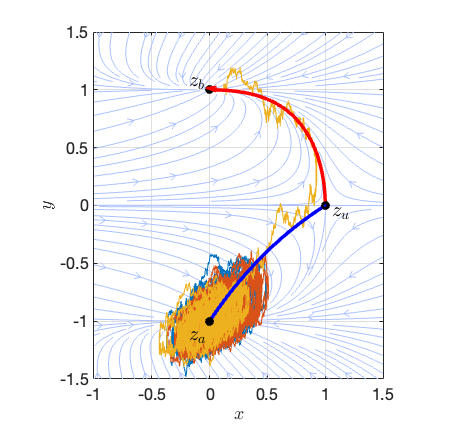}
\centerline{$(a)$}
\end{minipage}%
\hspace{-0.5cm}
\begin{minipage}[t]{0.45\linewidth}
\centering
\includegraphics[height=5.8cm]{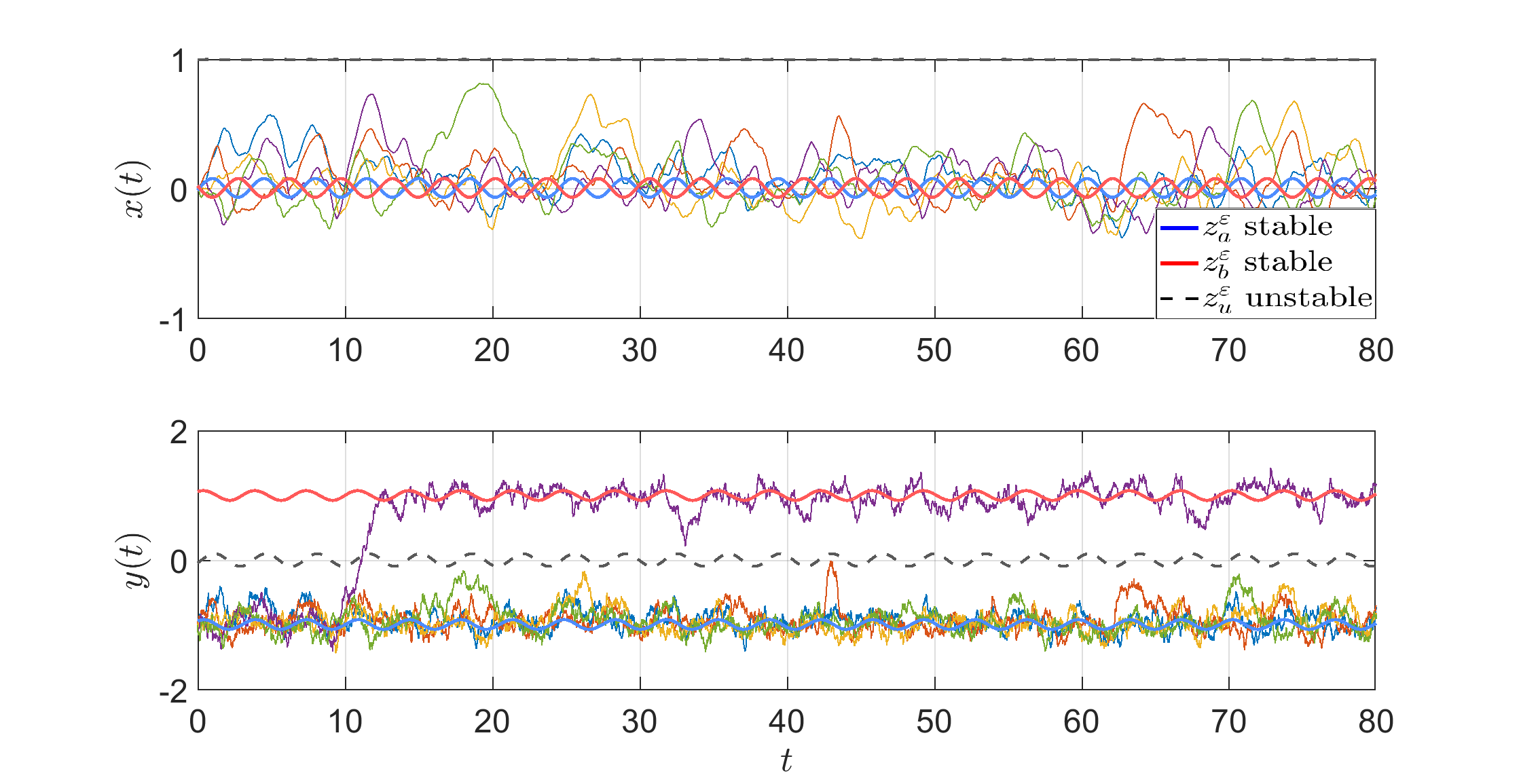}
\centerline{$(b)$}
\end{minipage}
\caption{(Color online) (a) \textit{Phase portrait, heteroclinic orbits, and sample paths of the stochastic system without periodic forcing.} The thin light-blue streamlines depict the phase portrait of the corresponding deterministic system. The thick blue curve is the heteroclinic orbit connecting the stable node $z_a$ to the saddle point $z_u$, while the red curve is the heteroclinic orbit connecting $z_u$ to the stable node $z_b$. The colored curves (10 in total) represent sample paths of the stochastic system for $\varepsilon=0$ and $\mu=0.25$. (b) \textit{Noise-induced transitions between metastable periodic orbits.} The upper panel displays sample paths of $x$, while the lower panel displays the corresponding paths of $y$. Here, $\mu=0.08$, $\varepsilon=0.2$, $\omega=1.8$, and $\vartheta=0$.
}\label{Fig1}
\end{figure}

\text{\textit{Verification of assumed conditions for the persistence of heteroclinic orbits}}\par
For $\varepsilon=0$, it is easy to see that the associated Hamiltonian system of \eqref{orthogonal_1}
admits an ``uphill" heteroclinic orbit \[
\Phi_{up}^0(t)
=
\left(
1-e^{-t}\arctan(e^t),
-\frac{1}{\sqrt{1+e^{2t}}},
0,
\frac{2e^{2t}}{(1+e^{2t})^{3/2}}
\right)^{\top},
\]
satisfying boundary conditions $\Phi_{up}^0(-\infty)=(0,-1,0,0)^{\top},\;\Phi_{up}^0(+\infty)=(1,0,0,0)^{\top}$ and admits a unique bounded conormal covector \[
v(t)
=
\nabla H\bigl(\Phi^0_{up}(t)\bigr)
=
\left(
0,\,
\frac{2e^{2t}(e^{2t}-2)}{(1+e^{2t})^{5/2}},\,
e^{-t}\arctan(e^t)-\frac{1}{1+e^{2t}},\,
\frac{e^{2t}}{(1+e^{2t})^{3/2}}
\right)^{\top}.
\] satisfying $\dot{v}(t)=-D\mathcal F(\Phi_{up}^0)^{\top}v(t)$. Thus, it is easy to prove Assumption~\textbf{(A3)} ($iii$).

Then the associated Melnikov function reads:\[
M_{up}(\alpha)
=
\int_{-\infty}^{+\infty}
\frac{2e^{2s}(2-e^{2s})}{(1+e^{2s})^{5/2}}
\cos(\omega(s+\alpha)+\vartheta)\,ds,
\]
here we take $h(z)=1$ for simplicity. Then by a direct Fourier transform computation,
\begin{align*}
M_{\mathrm{up}}(\alpha)
=
|F(\omega)|
\cos\bigl(\omega\alpha+\vartheta+\arg F(\omega)\bigr),
\end{align*}
where $F(\omega)
=
-i\omega
\frac{2}{\sqrt{\pi}}
\Gamma\left(1+\frac{i\omega}{2}\right)
\Gamma\left(\frac12-\frac{i\omega}{2}\right).$ 
For $\omega\neq0$, $|F(\omega)|>0$, there exists a constant $\alpha_0\in\mathbb R$ such that
\begin{align}
M_{\mathrm{up}}(\alpha_0)=0,
\qquad
M_{\mathrm{up}}'(\alpha_0)\neq0.\nonumber
\end{align}
Therefore, Assumption \textbf{(A3)} is verified for the orthogonal system.
Consequently, for sufficiently small periodic perturbations, there exists a perturbed heteroclinic orbit near the original one.

Figure \ref{Fig1}(b) shows a noise-induced transition under periodic forcing. 
The equilibria $z_a$, $z_b$, and $z_u$ become periodic orbits 
$z_a^\varepsilon$, $z_b^\varepsilon$, and $z_u^\varepsilon$, respectively. 
Most sample paths stay near the stable periodic orbit $z_a^\varepsilon$ for a long time. 
However, under small noise, some paths escape from $z_a^\varepsilon$, pass near the unstable periodic orbit $z_u^\varepsilon$, and finally move towards the stable periodic orbit $z_b^\varepsilon$. 
Thus, this figure illustrates a metastable transition between two stable periodic orbits.

\text{\textit{Variation of the Transition Rate}}

A direct calculation shows that $y$-component of the uphill heteroclinic orbits can be expressed as $y(t)=-(1+e^{2t})^{-\frac12}$.
For the case of $h(z)=1$, we obtain \[
\int_{-\infty}^{\infty} \dot{y}(t)e^{i\omega t} dt
= \frac{\Gamma\left(\frac{1-i\omega}{2}\right)\Gamma\left(1+\frac{i\omega}{2}\right)}{\sqrt{\pi}}.
\]
Then the correction $\delta S_e=-\frac2{\sqrt\pi}|\Gamma\left(\frac{1-i\omega}{2}\right)\Gamma\left(1+\frac{i\omega}{2}\right)|$.
\begin{figure}[htpp]
    \centering
    \includegraphics[width=0.5\linewidth]{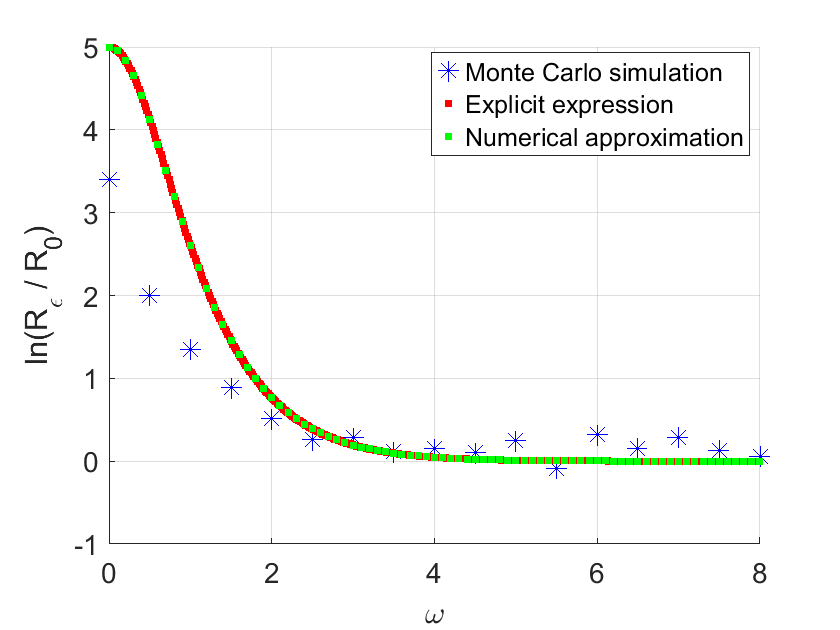}
    \vspace{-0.1cm}
    \caption{(Color online) \textit{Transition rate predicted by the asymptotic formula \eqref{rate-1} as a function of the forcing frequency $\omega$.} The solid curve is obtained from the explicit expression for $\delta S_e$, the dotted curve from numerical evaluation of the integral via piecewise trapezoidal quadrature, and the asterisks from Kramers' method combined with Monte Carlo simulations. Here, $\mu=0.08$ and $\varepsilon=0.2$.}
\label{fig:MC}
\end{figure}

For the case that explicit expressions of $\delta S_e$ is available,  we can also compute $\delta S_e$ by both the explicit expressions of $\delta S_e$ and by numerically approximating the integral in \eqref{delta}. Moreover,  by Theorem~\ref{the1}, the transition rate $\mathcal R_\varepsilon$ from one stable periodic orbit $z_a^\varepsilon(t)$ to another stable periodic orbit $z_b^\varepsilon(t)$ is
$\mathcal R_\varepsilon = C \exp\{-S^\varepsilon/\mu\} \approx \mathcal R_0 \exp\{-\varepsilon \delta S_e/\mu\},$ to first order in $\varepsilon$ or, equivalently,
\begin{equation}\label{rate-1}
\ln \frac{\mathcal R_\varepsilon}{\mathcal R_0} = - \frac{\varepsilon}{\mu} \delta S_e + \frac{1}{\mu} O(\varepsilon^2).
\end{equation}
where $\mathcal R_0 = C \exp\{-1/2\mu\}$ is the transition rate for the autonomous system \eqref{orthogonal_1} with $\varepsilon=0$, and $C$ is a positive constant. 
In figure \ref{fig:MC}, the Monte Carlo simulations, explicit expression and numerical approximation of the expression $\ln(\mathcal R_\varepsilon /\mathcal R_0)$ with respect to the transition rate are presented, where $\mu = 0.08,\;\varepsilon=0.2$.

\paragraph{Example B: A simplified FitzHugh--Nagumo system} 

The FitzHugh–Nagumo model is a well-known two-variable nonlinear system that has been studied in~\cite{FitzHugh1961,Gong2001}. Here, we consider the following simplified system:
\begin{align}
\begin{cases}\label{case1}
dx = \big(x - \tfrac{1}{3}x^3 - y\big)\,dt,\\[4pt]
dy = \big[-y+\varepsilon h(x,y)\cos(\omega t+\vartheta)\big]\,dt
+ \sqrt{\mu}\,dB_t.
\end{cases}
\end{align}

It is easy to check that this system satisfies Assumptions \textbf{(A1)--(A2)}, and that the Lyapunov conditions hold with
$\mathcal U(x,y)=\frac{1}{12}(x^2-3)^2+\frac12 y^2.$
When \(\varepsilon=\mu=0\), the unperturbed deterministic system has two stable equilibria $z_a=(-\sqrt 3,0)^{\top},\; z_b=(\sqrt 3,0)^{\top},$ and one saddle point $z_u=(0,0)^{\top}$.

\textit{Numerical representation of the heteroclinic orbit for $\varepsilon=0$.}

Unlike Example A, the uphill heteroclinic orbit of this system has no explicit expression in the two-dimensional configuration space. Therefore, we compute it numerically in the phase space. By section~\ref{sec3}, the heteroclinic orbit of the original system can be characterized by the corresponding Hamiltonian system. We first solve the Hamiltonian boundary value problem using MATLAB's \texttt{bvp4c}, and then project the obtained orbit onto the $xOy$-plane.

We now add the conjugate momentum variables
$(p_x,p_y)$. Then the associated Hamiltonian  system is
\begin{equation}\label{h22}
\begin{cases}
\dot x=x-\dfrac13x^3-y,\\[6pt]
\dot y=-y+p_y,\\[6pt]
\dot p_x=(x^2-1)p_x,\\[6pt]
\dot p_y=p_x+p_y,
\end{cases}
\end{equation} 
where the ideal infinite-time boundary conditions are
$
(x,y,p_x,p_y)(-\infty)
=
(-\sqrt3,0,0,0),$
and
$(x,y,p_x,p_y)(+\infty)=(0,0,0,0).$

In the numerical computation, the infinite-time boundary conditions are
approximated on a large finite interval $[-T,T]$ by imposing boundary
conditions on the corresponding unstable eigenspace at $\Gamma_a$ and
stable eigenspace at $\Gamma_u$. Its projection onto the $(x,y)$-plane gives the uphill heteroclinic of the MLP.
The downhill segment from $z_u$ to $z_b$ is obtained from the deterministic system
$\dot x=x-\frac13 x^3-y,\;\dot y=-y$. Figure~\ref{fig:sample_hetercolinic} shows that the sample paths tend to
concentrate near the computed MLP. 
\begin{figure}[htbp]
    \centering
\includegraphics[width=0.6\linewidth]{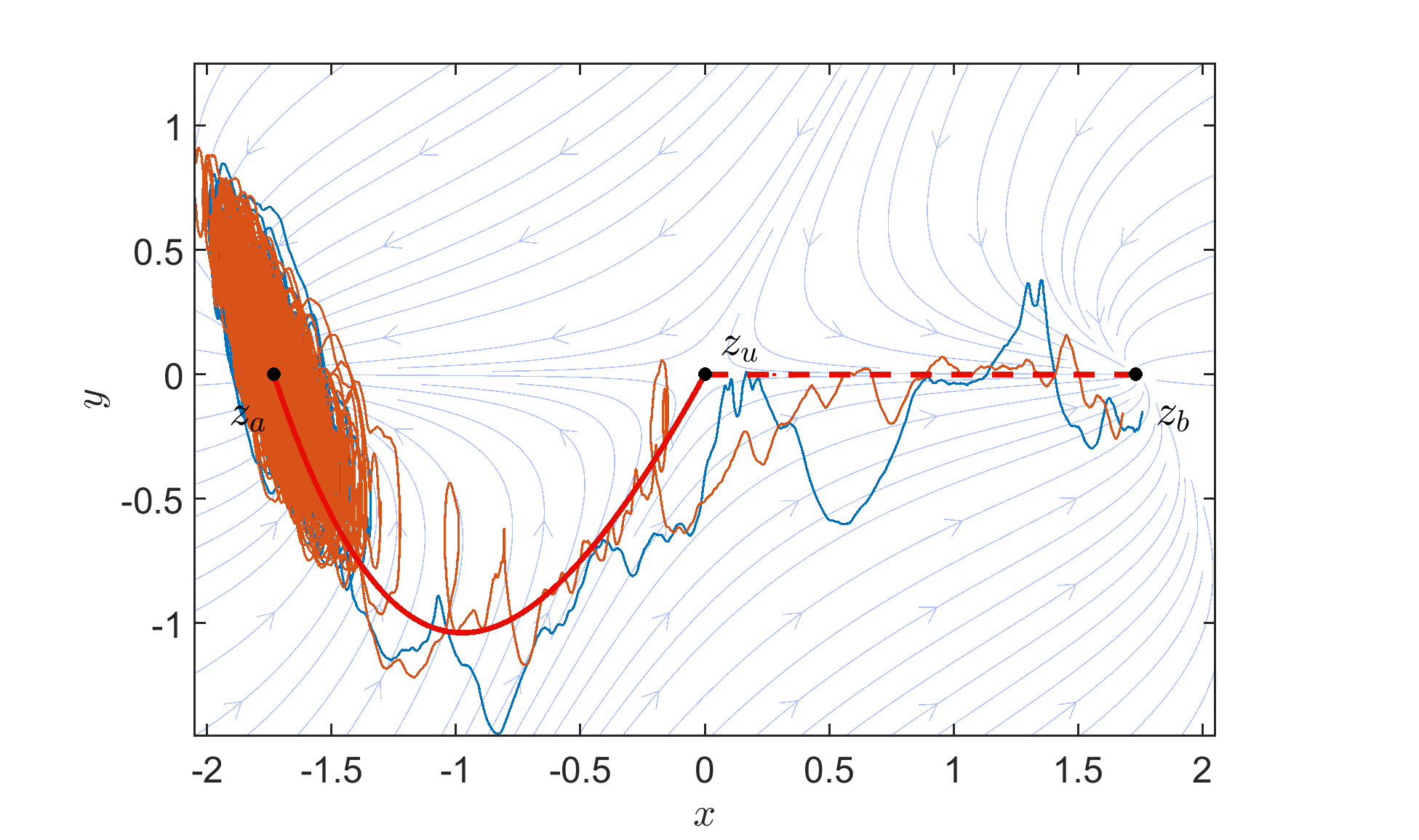}
\vspace{-0.2cm}
    \caption{(Color online) \textit{Phase portrait, heteroclinic orbits, and sample paths of the stochastic system without periodic forcing.} The thin light-blue streamlines depict the phase portrait of the deterministic system. The solid black curve is the MLP from $z_a$ to $z_u$, and the dashed red curve is the MLP from $z_u$ to $z_b$. The colored curves show representative sample paths for $\varepsilon=0$ and $\mu=0.4$; among 30 simulated trajectories, only two successful transitions from $z_a$ to $z_b$ via $z_u$ are displayed.}
\label{fig:sample_hetercolinic}
\end{figure}

Alternatively, the MLP in this case can be characterized through the quasi-potential $V$~\cite{liu2022}, which governs the rare transition dynamics and is defined as the minimal Freidlin--Wentzell action required to reach a point $z=(x,y)$ from a reference attractor:
\begin{equation}
V(z) = \inf_{\substack{\phi \in \mathscr{A} \\ \phi(-\infty) = z_a, \ \phi(+\infty) = z}} S^0_{\mathrm{FW}}[\phi].
\end{equation}
Rather than directly minimizing the action functional or solving the associated Hamiltonion equation, we compute $V$ numerically via the stationary Fokker--Planck equation. According to the large deviation principle, the stationary density $\rho_\mu$ satisfies the exponential asymptotic
$$
\rho_\mu(z) \sim \exp\!\left(-\frac{V(z)}{\mu}\right), \qquad \mu \to 0,
$$
which directly implies $V(z) = -\mu \log \rho_\mu(z) + O(\mu)$. For the present two-dimensional system, the stationary Fokker--Planck equation \cite{Duan2015} reads
\begin{equation}\label{FP-eq}
\mathcal{A}^* \rho_\mu = 0,
\end{equation}
where the Fokker--Planck operator $\mathcal{A}^*$ acting on $\rho_\mu$ is explicitly given by
$$
\mathcal{A}^* \rho_\mu
= -\frac{\partial}{\partial x}\Big[\big(x - \tfrac{1}{3}x^3 - y\big)\rho_\mu\Big]
 -\frac{\partial}{\partial y}\big[-y\,\rho_\mu\big]
 + \frac{\mu}{2}\frac{\partial^2 \rho_\mu}{\partial y^2}.
$$
We numerically solve this equation by discretizing the state space on a uniform grid and approximating the operator using an upwind Markov-chain scheme. The resulting stationary density $\rho_\mu$ is obtained from the null eigenvector of the discretized operator, and is then smoothed to suppress numerical artifacts. The quasi-potential is finally approximated by
\begin{equation*}
V(x,y) \approx -\mu \log \rho_\mu(x,y),
\end{equation*}
normalized by subtracting its global minimum. This quasi-potential landscape provides a numerical consistency check
for the computed MLP; see~figure~\ref{quasi-potential}.

\begin{figure}[htbp]
\begin{minipage}[t]{0.5\linewidth}
\centering
\includegraphics[height=6cm]{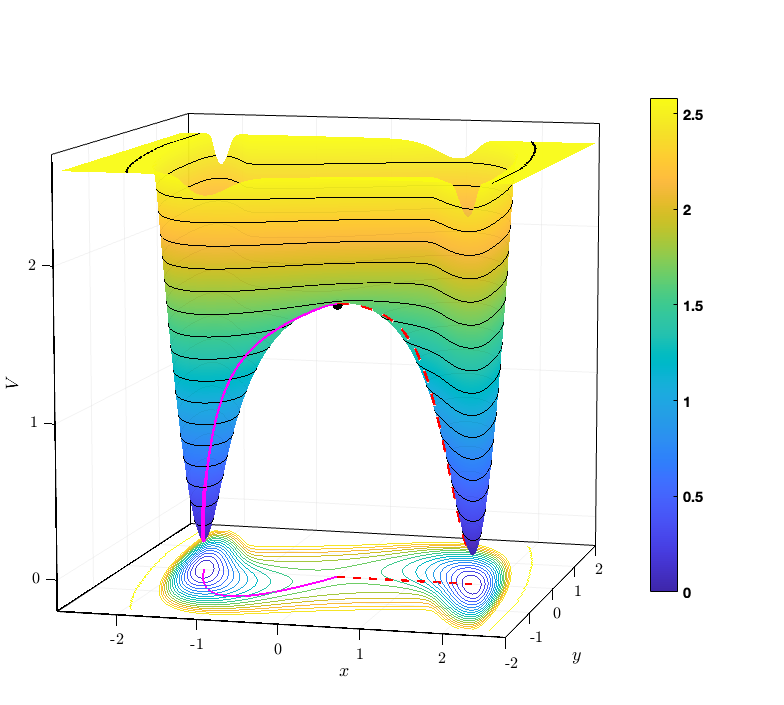}
\centerline{$(a)$}
\end{minipage}%
\begin{minipage}[t]{0.5\linewidth}
\centering
\includegraphics[height=5.5cm]{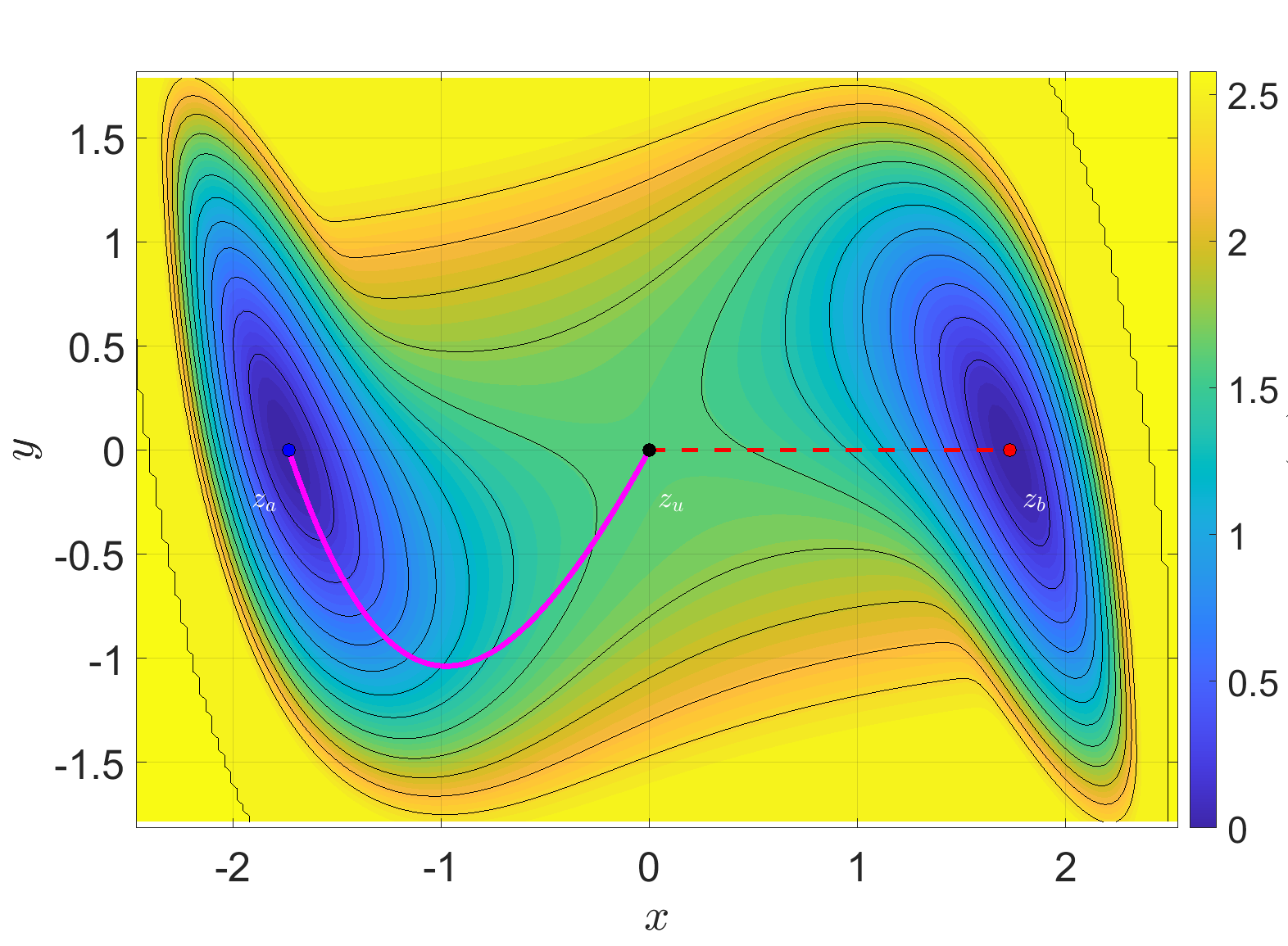}
\centerline{$(b)$}
\end{minipage}
\caption{(Color online) \textit{Quasi-potential landscape and MLP with $\varepsilon=0$ and $\mu=0.08$.} Filled contours show $V(x,y)$ computed from the stationary Fokker--Planck equation \eqref{FP-eq}; black lines are level sets. The magenta curve is the MLP from $z_a$ to $z_u$, obtained by action minimization. The red dashed curve is the deterministic unstable manifold from $z_u$ to $z_b$, which lies on the $x$-axis since $y(0)=0$ and $\dot y=-y$. Dots mark the equilibria $z_a$, $z_u$, and $z_b$.}
\label{quasi-potential}
\end{figure}

\text{\textit{Numerical validation of metastable transitions under periodic forcing.}}

We further consider the persistence of MLP under small perturbations. From above results, we obtain the existence of the unperturbed uphill orbit
$\Phi_{\mathrm{up}}^0(t)$ numerically. For this example, the regularity assumptions are straightforward.
The remaining key conditions in Assumption \textbf{(A3)} are the transversality
of the unperturbed heteroclinic connection and the nondegeneracy of the corresponding Melnikov function. Note that in phase space, the function $g(t,x,y)=(0,h(x,y)cos(\omega t+\vartheta))^{\top}$ becomes \[
\mathcal G(t,\Phi)
=
\begin{pmatrix}
0\\
h(x,y)\\
-p_y h_x(x,y)\\
-p_y h_y(x,y)
\end{pmatrix}
\cos(\omega t+\vartheta):=\mathcal G_0(\Phi)\cos(\omega t+\vartheta).
\] Assume $$T_{\Phi_{\mathrm{up}}^0(t)}W^u(\Gamma_a)
\cap
T_{\Phi_{\mathrm{up}}^0(t)}W^s(\Gamma_u)=
\operatorname{span}\{\dot{\Phi}_{\mathrm{up}}^0(t)\},$$ 
and
$$ \int_{-\infty}^{+\infty}
K(s)e^{i\omega s}\,ds
\neq 0,$$
where $\Gamma_a=(-\sqrt3,0,0,0)^{\top},\;\Gamma_u=(0,0,0,0)^{\top}$
and 
$K(s)
=
\nabla_{\Phi}H_0\bigl(\Phi_{\mathrm{up}}^0(s)\bigr)^{\top}
\mathcal G_0\bigl(\Phi_{\mathrm{up}}^0(s)\bigr).$
 Then, under these assumptions,  \eqref{case1} admits an MLP that persists under sufficiently small perturbations. 

\begin{figure}[htbp]
\begin{minipage}[t]{0.5\linewidth}
\centering
\includegraphics[height=4.2cm]{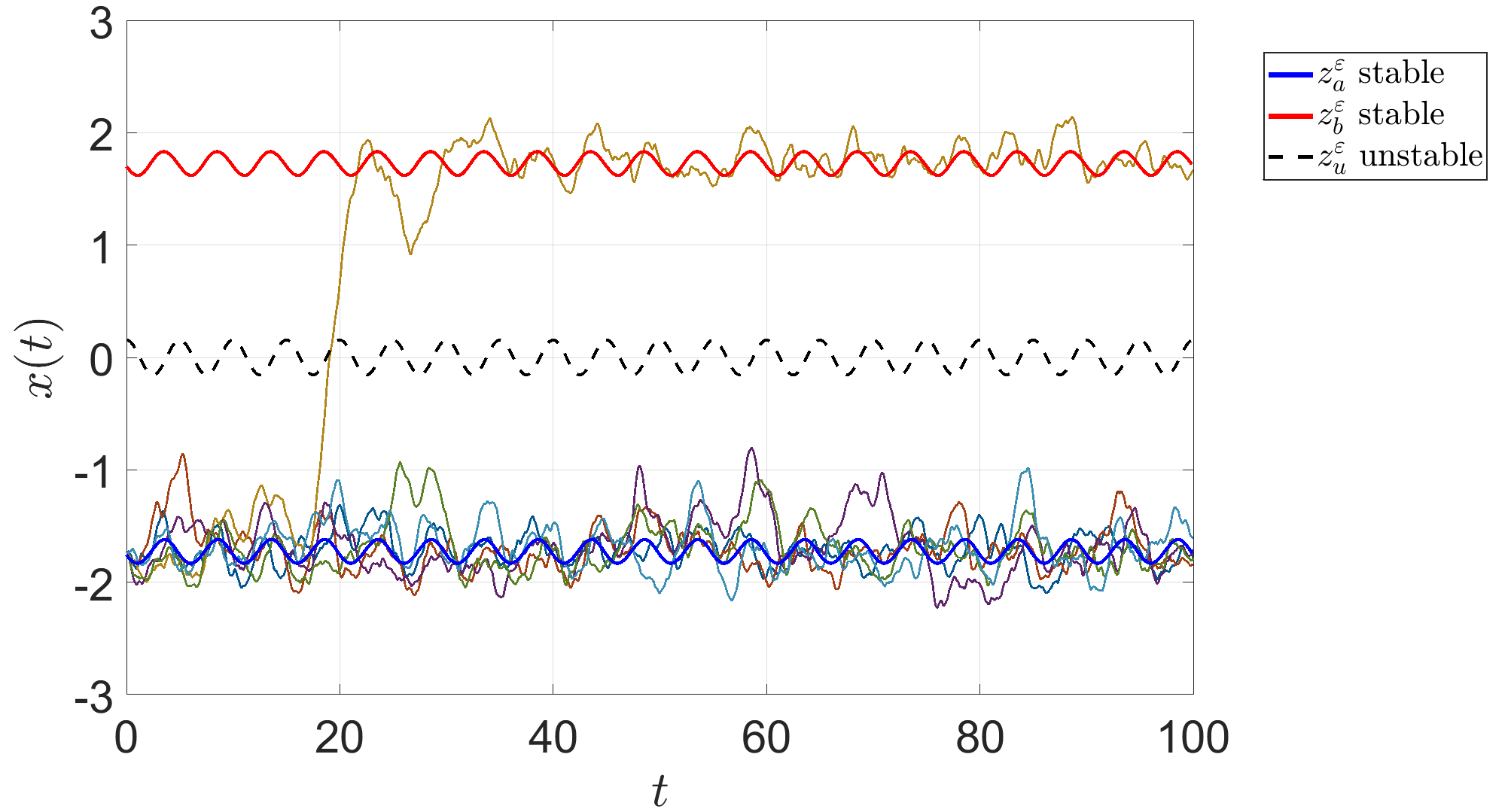}
\centerline{$(a)$}
\end{minipage}%
\begin{minipage}[t]{0.5\linewidth}
\centering
\includegraphics[height=4.2cm]{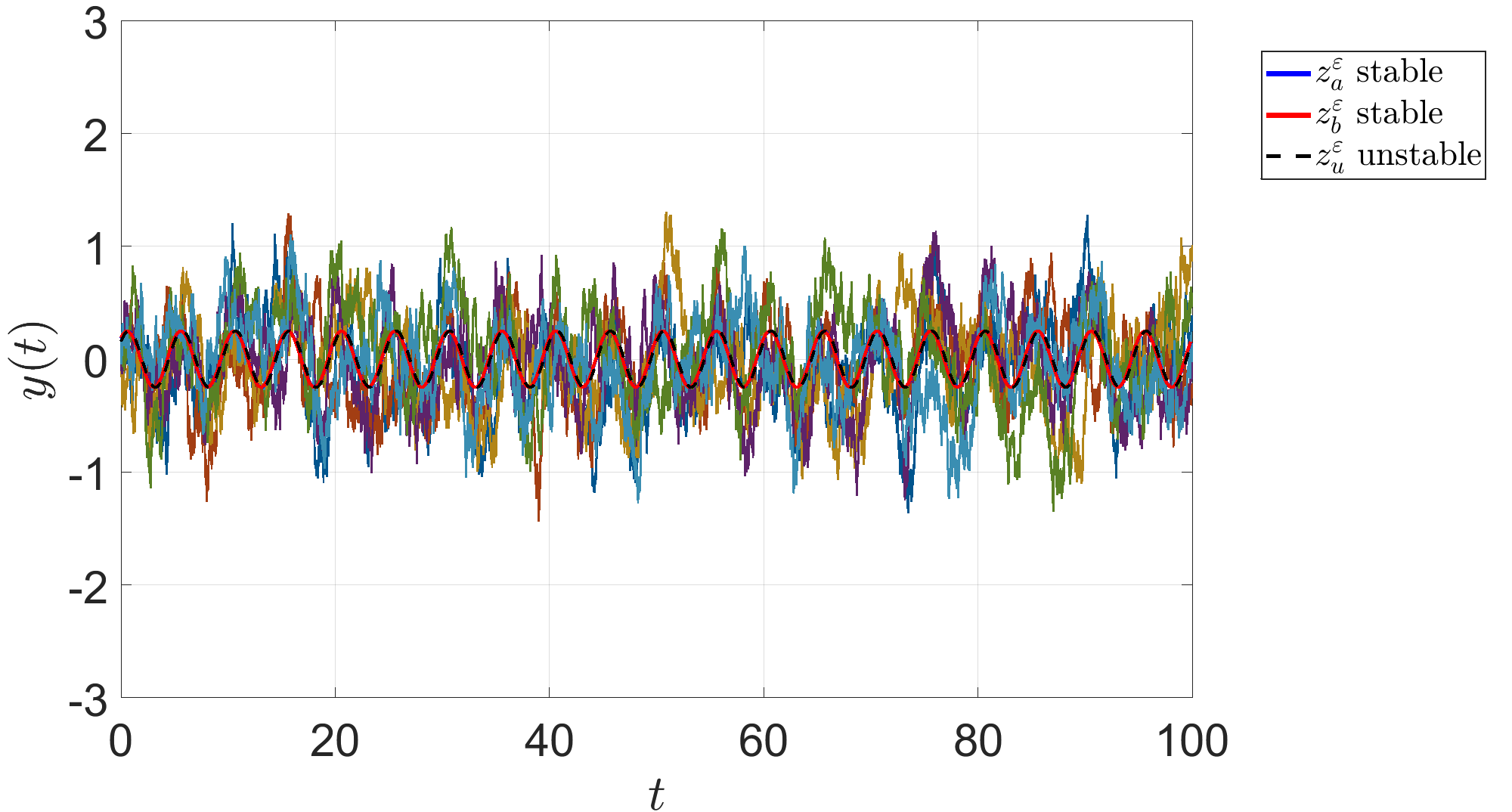}
\centerline{$(b)$}
\end{minipage}
\caption{(Color online) \textit{Noise-induced transitions between metastable periodic orbits.}
(a) Projection onto the $x$-coordinate. 
(b) Projection onto the $y$-coordinate. 
The solid blue and red curves denote the stable periodic orbits $\gamma_a^\varepsilon(t)$ and $\gamma_b^\varepsilon(t)$, respectively, while the black dashed curve represents the unstable periodic orbit $\gamma_u^\varepsilon(t)$. The colored paths show sample trajectories (6 in total) under stochastic forcing, exhibiting noise-induced transitions between the two stable periodic orbits. 
Here: $\mu = 0.25$, $\varepsilon = 0.4$, $\omega = 2\pi/5$, $\vartheta = 0$.}\label{fig:two-figures}
\end{figure}

Figure~\ref{fig:two-figures} shows the projections of the noise-induced transitions between stable periodic solutions onto the $x$- and $y$-coordinates for $\mu=0.25,\;\varepsilon=0.4,\;\omega=2\pi/5,\;\vartheta=0$. To better display the periodic structure, we represent the dynamics on a cylinder in~figure~\ref{11111}.

\begin{figure}[htbp]
\begin{minipage}[t]{0.5\linewidth}
\centering
\includegraphics[height=4.2cm]{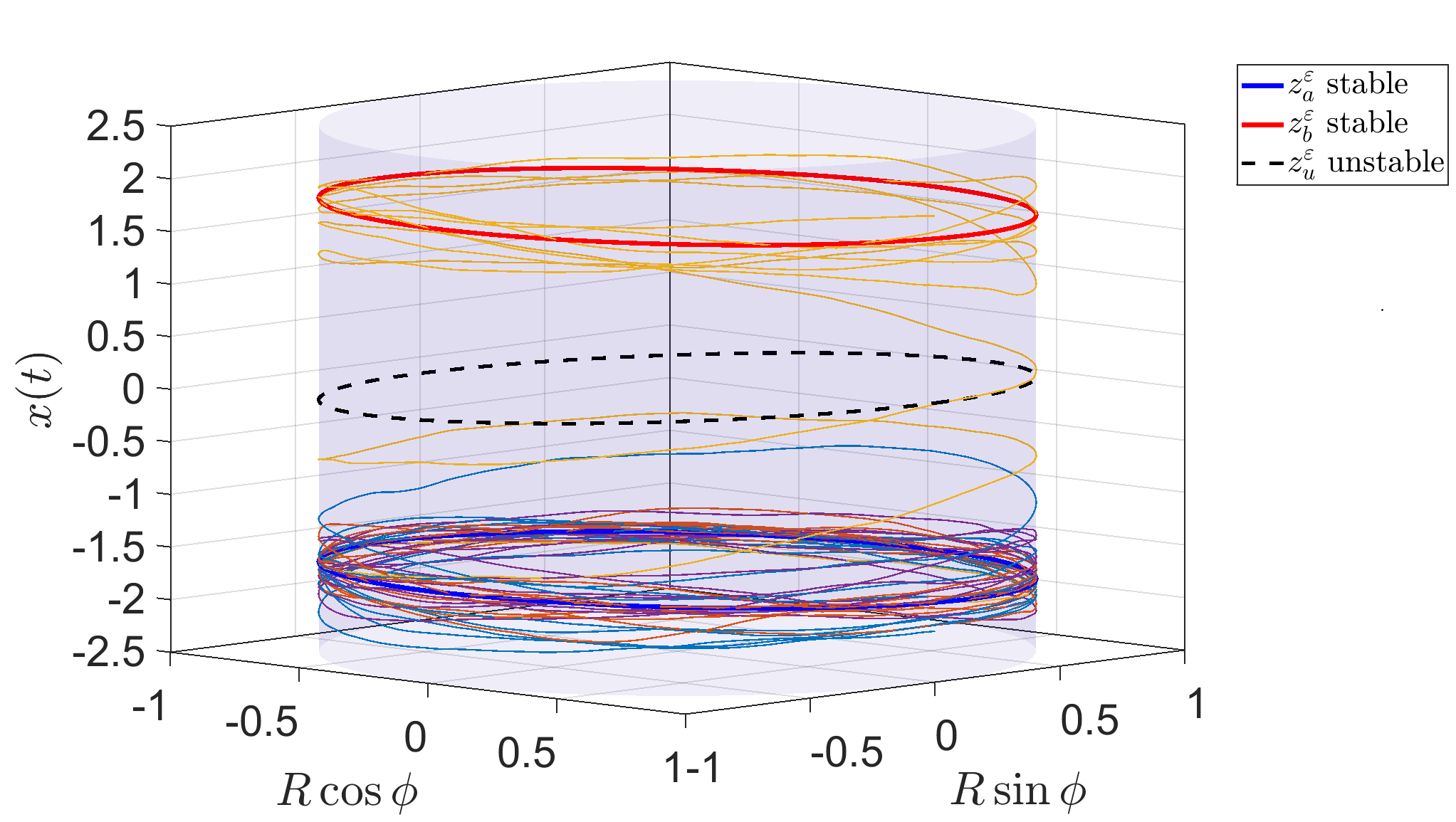}
\centerline{$(a)$}
\end{minipage}%
\begin{minipage}[t]{0.5\linewidth}
\centering
\includegraphics[height=4.2cm]{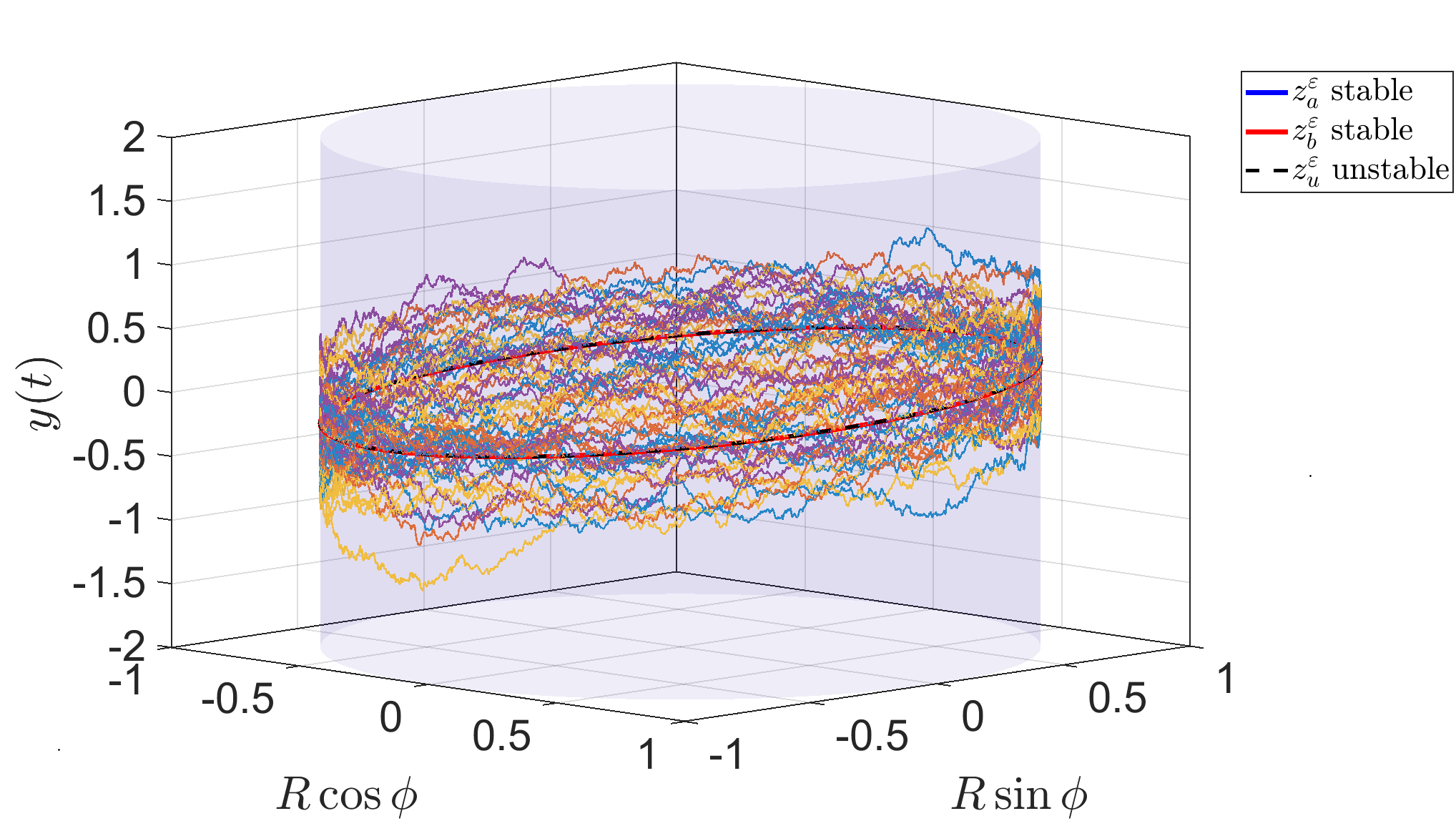}
\centerline{$(b)$}
\end{minipage}
\caption{(Color online) \textit{Cylindrical projections of noise-induced transitions between metastable periodic orbits.} Same parameter setting as in in Figure~\ref{fig:two-figures}. Solid blue and red rings in (a) show the lower and upper stable periodic orbits of $x(t)$ for \eqref{case1} with $\mu=0$, with colored curves (8 in total) indicating noise-induced transition paths between them. In (b) the stable and unstable periodic orbits overlap, and the colored curves (8 in total) illustrate its noise-driven dynamics under direct Brownian forcing.}\label{11111}
\end{figure}

Unlike Example A, the present FitzHugh--Nagumo type system does not admit an explicit expression for the uphill MLP or for the action correction $\delta S_e$. Hence the transition rate formula cannot be compared with an explicit analytical benchmark. Moreover, the noise is degenerate and acts only on the $y$-component, so direct Monte Carlo estimation of rare transition rates would require a very large number of sample paths and long integration times. Therefore, instead of numerically estimating the transition rate, we focus on computing the unperturbed Hamiltonian heteroclinic orbit by a boundary value method and numerically validating metastable transitions under periodic forcing.

\section{Conclusion}
\label{sec7}
In this paper, we have established the persistence of the MLPs under small time-dependent forcing for a general class of higher-dimensional degenerate stochastic systems. Under suitable hyperbolicity and transversality assumptions, we proved the persistence of the uphill heteroclinic connection by a geometric Melnikov method, which does not require the perturbation to be time-periodic.

For periodic perturbations, we derived a closed-form explicit expression that characterizes how a small, generic, nonlinear periodic forcing affects the metastable transition rate. Moreover, the rate of the metastable transition is also approximated in terms of the Melnikov function. The examples illustrate how the periodic forcing can either enhance or suppress metastable transitions depending on the forcing phase and frequency. These results provide a geometric mechanism for understanding noise-induced transitions in degenerate stochastic systems with weak time-dependent perturbations.

\bibliographystyle{alpha}
\bibliography{Wei-refs}
\end{document}